\documentclass[reqno]{amsart}
\usepackage{amssymb,latexsym,color,}
\usepackage{amsmath}
\usepackage{amsthm}
\usepackage{graphicx}
\usepackage{palatino,mathpazo}
\usepackage{color}
\usepackage{titletoc}
\numberwithin{equation}{section}
\newtheorem{theorem}{Theorem}[section]
\newtheorem{proposition}[theorem]{Proposition}
\newtheorem{lemma}[theorem]{Lemma}

\newtheorem{Theorem}{Theorem}[section]

\theoremstyle{definition}

\def\t{\tilde}

\def\XXint#1#2#3{{\setbox0=\hbox{$#1{#2#3}{\int}$}
     \vcenter{\hbox{$#2#3$}}\kern-.5\wd0}}

\def\l{\langle}
\def\r{\rangle}
\def\t{\tilde}

\def\p{\partial}

\def\pai{\mbox{\boldmath$\pi$}}
\newcommand{\R}{\mathbb{R}}

\newcommand{\e}{\varepsilon}

\begin{document}
\title[collapsing singularity]{Sharp estimates for solutions of mean field equation with collapsing singularity}
\author{Youngae Lee}
\address{Youngae ~Lee,~National Institute for Mathematical Sciences, 70 Yuseong-daero 1689 beon-gil, Yuseong-gu, Daejeon, 34047, Republic of Korea}
\email{youngaelee0531@gmail.com}
\author{ Chang-shou Lin}
\address{ Chang-shou ~Lin,~Taida Institute for Mathematical Sciences, Center for Advanced Study in
Theoretical Sciences, National Taiwan University, Taipei 106, Taiwan}
\email{cslin@math.ntu.edu.tw}
\author{Gabriella Tarantello}
\address{Dipartimento di Matematica, Universit$\grave{a}$ di Roma "Tor Vergata", via della Ricerca Scientifica, 00133 Rome, Italy}
\email{tarantel@mat.uniroma2.it}
\author{Wen Yang}
\address{ Wen ~Yang,~Center for Advanced Study in Theoretical Science,National Taiwan University, No.1, Sec. 4, Roosevelt Road, Taipei 106, Taiwan}
\email{math.yangwen@gmail.com}

\date{}
\begin{abstract}
The pioneering work of Brezis-Merle \cite{bm}, Li-Shafrir \cite{ls}, Li \cite{l} and Bartolucci-Tarantello \cite{bt1} showed that any sequence of blow up  solutions for (singular) mean field equations of Liouville type must exhibit  a  "mass concentration" property. A typical situation of blow-up occurs when we let the singular (vortex) points involved in the equation (see  (\ref{0.0}) below)  collapse together.  However  in this case Lin-Tarantello in \cite{lt}  pointed out that the phenomenon:  "bubbling implies mass concentration"  might not occur and new scenarios open for investigation. In this paper, we present two explicit examples which illustrate (with mathematical rigor) how a "non-concentration" situation does happen and its new features.  Among other facts, we show that in certain situations, the collapsing rate of the singularities can be used as blow up parameter to describe the bubbling properties of the solution-sequence. In this way we are able to establish accurate estimates around the blow-up points which we hope to use towards a degree counting formula for the shadow system (\ref{1.14}) below.
\end{abstract}

 \maketitle

\section{Introduction}
In this paper, we wish to discuss the blow-up behaviour for solutions of the following mean field equation of Liouville type:
\begin{align}
\label{0.0}
\Delta u^*+\rho\left(\frac{h^*e^{u^*}}{\int_Mh^*e^{u^*}}-\frac{1}{|M|}\right)
=4\pi\sum_{i=1}^N\alpha_i(\delta_{q_i}-\frac{1}{|M|})~\mathrm{in}~M,
\end{align}
where $(M,g)$ is a compact Riemann surface and $|M|$ is its area. For simplicity we assume the area $|M|=1.$ Here $\Delta$ stands for the Beltrami-Laplacian operator on $(M,g)$ and $q_1,q_2,\cdots,q_N$ are given distinct points in $M.$ For convenience, we let $S=\{q_1,q_2,\cdots,q_N\}.$ Throughout this paper, we always assume that $h^*(x)$ is a positive $C^2$ function in $M$ and $u^*\in\mathring{H}^1(M),$ where
$$\mathring{H}^1(M):=\{f\in H^1(M)\mid \int_Mf=0\}.$$
Equation (\ref{0.0}) arises in many different areas of mathematics and physics. On the flat torus, the following singular Liouville equation
\begin{align}
\label{0.1}
\Delta u+e^{u}=\rho\delta_0
\end{align}
has attracted a lot of attention in the past and recent years. By integrating equation (\ref{0.1}), we can easily re-stated it equivalently as the following mean field equation:
\begin{align*}
\Delta u+\rho\left(\frac{e^u}{\int_Te^u}-\frac{1}{|T|}\right)=\rho(\delta_0-\frac{1}{|T|})~\mathrm{on}~T.
\end{align*}
In geometry, equation (\ref{0.1}) is related to the prescribed Gaussian curvature problem. Indeed, we may consider more generally the following equation:
\begin{align}
\label{0.2}
\Delta u+he^{u}-2k=4\pi\sum_{i=1}^N\alpha_i\delta_{q_i},
\end{align}
where $k$ is the constant Gaussian curvature of the given metric $g$ and $h(x)$ is a positive function on $M.$ For any solution $u$ of (\ref{0.2}), we obtain the new metric $\tilde g:=e^{v}g$ (where $v=u-\log 2$) with Gaussian curvature $\tilde{k}(x)=h(x)$ away from the singular points $q_i$, which represent conical singularities for $(M,\tilde{g})$. Again, by integrating the equation (\ref{0.2}), we can easily see that $u(x)$ satisfies equation (\ref{0.0}) with $\rho=2k+4\pi\sum_{i=1}^N\alpha_i$. Hence, we can regard (\ref{0.2}) as a particular case of (\ref{0.0}). When $M=\mathbb{S}^2$ and $\alpha_i=0 ~\forall i=1,...,N,$ then equation (\ref{0.2}) is related to the well-known Nirenberg problem. Moreover, if the parameter $\rho=4\pi\sum_{i=1}^N\alpha_i$ and $M$ is a flat surface (for example, a flat torus), then equation (\ref{0.0}) can be viewed as an integrable system, related to the classical Lame equation and the Painleve VI equation (see \cite{clw,cklw} for details). In physics, equation (\ref{0.0}) is related to the self-dual equations governing vortices for the relativistic Chern-Simons-Higgs model, and in this context the Dirac poles $q_i$ $i=1,\cdots,N$ represent the distinct \textit{vortex points} with multiplicity $\alpha_i \in \mathbb{N}$. We refer the readers to \cite{bamar, bmar, bt1, bclt, bm, cgy, ck,cl1, cl2, cl3, cl4,  l1, lw0,  ly1,  m1, m2, nt1, nt2,t, y,y2} for details.

In the seminal work \cite{bm}, Brezis and Merle initiated the study of the blow up behavior of solutions for Liouville-type equations. On the basis of the results in \cite{bm}, it is possible to derive that blow-up is always associated to a "concentration" property. More precisely, the following holds:
\medskip

\noindent {\bf{Theorem A}}. \cite{bm}  {\em  Suppose $S=\emptyset$ (or equivalently $\alpha_i =0,  \forall i=1,...,N$) $h^*$ is a positive smooth function and let $u^*_k$ be a sequence of blow up solutions for \eqref{0.0}, that is: $\text{max}_M\ u_k^* \rightarrow +\infty$ as $k \rightarrow +\infty$. Then there is a non- empty finite set $\mathcal{B}$ (blow up set)
such that,
 \[\rho\frac{h^*e^{u_k^*}}{\int_Mh^*e^{u_k^*}\mathrm{d}v_g}\to\sum_{p\in \mathcal{B}}\beta_p\delta_p,\ \ \textit{where}\ \ \beta_p\ge 4\pi.\]}

\indent The authors in \cite{bm} conjectured that the local mass $\beta_p\in 8\pi\mathbb{N}$ for any $p\in\mathcal{B},$ ( $\mathbb{N}$ is the set of natural numbers). This fact was proved by Li and Shafrir \cite{ls} in a more general setting. Actually, in the context of Theorem A, the authors in \cite{ls} showed that the local mass $\beta_p$ equals $8\pi$ exactly. In case $S\neq\emptyset$,  i.e. equation (\ref{0.0}) includes Dirac measures supported at each point in $S$, then Bartolucci and the third author of this article, in \cite{bt1} extended Theorem A by showing that, when a blow up point in $\mathcal{B}$ coincides with some $q_i\in S$, then the "concentration" properties stated above continues to hold with $\beta_{q_i}=8\pi(1+\alpha_i)$.\\
For equation \eqref{0.0}, we call $\rho\frac{h^*e^{u^*}}{\int_Mh^*e^{u^*}\mathrm{d}v_g}$ the \textit{mass distribution} of the solution $u^*$. Following this terminology, then Theorem A states that: if the solution sequence $u_k^*$ blows up then its mass distribution is {\em concentrated}.\\

%%%%%%%%%%%%%%%%%%%%%%%%%%%%%%%%%%%%%%%%% INSERT 1 %%%%%%%%%%%%%%%%%%%%%%%%%%%%%%
One of the most physically interesting situation of blow-up occurs when we let different vortex points (i.e. the Dirac poles in (\ref{0.0})) collapse into each other. In this situation, blow-up is registered also by topological considerations. Indeed, we observe a "jump" on the value of the Leray-Schauder degree associated to (\ref{0.0}), when we let different $q_i's$ merge together, see \cite{cl4}. However, it was recently observed in \cite{lt} that, in contrast to \cite{bt1,bm}, such blow-up phenomenon is not characterized by the "concentration" property of its mass distribution. To be more precise, for $t \in (0,1)$ and $1\leq d\leq N$ we let $u_t^* \in \mathring{H}_{1}(M)$ satisfy:
%%%%%%%%%%%%%%%%%%%%%%%%%%%%%%%%%%%%%%%%%%%%%%%%%%%%%%%%%%%%%%%%%%%%%%%%%%%%%%%%%
\begin{equation}
\label {0.3}
\Delta u^*_t+\rho\left(\frac{h^*e^{u^*_t}}{\int_Mh^*e^{u^*_t}\mathrm{d}v_g}-1\right)=4\pi\sum_{i=1}^d\alpha_i(\delta_{q_i(t)}-1)+ 4\pi\sum_{i=d+1}^N\alpha_{i}(\delta_{q_i}-1)\ \textrm{on} \ M.
\end{equation}
Clearly the last summation in the right hand side of (\ref{0.3}) is included only when $d<N$.\\
Throughout the paper, we always assume that $\lim_{t\to0}q_i(t)= \mathfrak{q}\notin\{q_{d+1}, \cdots, q_N\}$, $\forall\ i = 1, \cdots , d$ and $q_i(t) \neq q_j(t)$ for $i \neq j \in \{ 1,\cdots,d \}$.  Then the following holds:
\medskip

\noindent {\bf{Theorem B.1}}. \cite{lt}  {\em  Assume that $\alpha_{i}\in\mathbb{N},~1\leq i\leq N$, $h^*$ is a smooth positive function, and $\rho\in(8\pi,16\pi)$.
Suppose that  $u_t^*$ is a blow up solutions of \eqref{0.3} as $t\to0$. Then $u_t^*\to w^*$ uniformly in $C_{\mathrm{loc}}^2(M\setminus\{\mathfrak{q}\})$, and $w^*$ satisfies, }
\begin{equation}\label{0.31}
\begin{aligned}
&\Delta w^*+(\rho-8\pi)\Big(\frac{h^*e^{w^*}}{\int_Mh^*e^{w^*}\mathrm{d}v_g}-1\Big) =4\pi(\sum\limits_{i=1}^{d}\alpha_i-2)(\delta_{\mathfrak{q}}-1)+4\pi\sum_{j=d+1}^N\alpha_{j}(\delta_{q_j}-1).
\end{aligned}
\end{equation}
An analogous phenomenon occurs also when we let $u_k^*$ to satisfy a "regularization" of problem (\ref{0.0}) in the sense that, the Dirac measures on the right hand side of (\ref{0.0}) are replaced by their convolution with smooth mollifiers. We point out that this situation describes a typical blow-up scenario in the context of Liouville systems, where much of the bubbling phenomenon still needs to be understood.

In \cite{lt} we find the following result,
\medskip

\noindent {\bf Theorem B.2.} {\em Let $\mathbb{S}^2$ be the standard $2$-sphere and let $u^*_k \in \mathring{H}^{1}(\mathbb{S}^2)$ satisfy:
\begin{equation}
\label{0.32}
\Delta u^*_k + \rho(\frac{e^{u^*_k}}{\int_{\mathbb{S}^2} e^{u^*_k}} - \frac{1}{4\pi}) = 4\pi \alpha(h_k - \frac{1}{4\pi})\ \ \text{on}\ \ \mathbb{S}^2,
\end{equation}
with (suitable) $ 0<h_k \in C^2(\mathbb{S}^2):\int_{\mathbb{S}^2}h_k = 1$ and $h_k \rightharpoonup \delta_p$, weakly in the sense of measure in $\mathbb{S}^2$ for some $p \in \mathbb{S}^2$. For $1<\alpha<3$, there exists $\rho_{\alpha}^*\in(4\pi(\alpha+1),8\pi\min\{2,\alpha\})$ such that if
\begin{equation}
\label{0.33}
\begin{aligned}
4\pi(\alpha+1)<\rho<\rho_{\alpha}^*
\end{aligned}
\end{equation}
then $u_k^*$ blows up only at $p$, namely along a subsequence the following hold:
\begin{equation*}
\begin{aligned}
&u_k^* (p_k)=\max\limits_{\mathbb{S}^2} u^*_k \rightarrow + \infty,\ \ p_k \rightarrow p,\\
u^*_k&\rightarrow u\ \ \text{in}\ \ C^2_{\mathrm{loc}}(\mathbb{S}^2\backslash\{p\}),\ \mbox{as}\ k \rightarrow +\infty,
\end{aligned}
\end{equation*}
and $u$ satisfies
\begin{equation}
\label{0.34}
\Delta u+(\rho-8\pi)(\frac{e^u}{\int_{\mathbb{S}^2}e^u}-\frac{1}{4\pi})=4\pi(\alpha-2)(\delta_p-\frac{1}{4\pi})
\ \ \text{on}\ \ \mathbb{S}^2.
\end{equation}}

The origin of the value $\rho_{\alpha}^*$ will become clear from Proposition \ref{pr2.new} and Remark 2.1 below. Notice that, by the assumptions of Theorem B.1 and Theorem B.2, we have: $0 < \rho-8\pi < 8\pi$; and so we can ensure that both problem (\ref{0.31}) and (\ref{0.34}) admit a solution on the basis of Onofri inequality. The authors in \cite{lt} just give a sketch of the proof of Theorem B.1 and Theorem B.2 and one of the purposes of this paper is to provide a detailed proof of those results together with some new facts.

For example, we show that the following general version of Theorem B.1 holds:

\begin{theorem}
\label{th1.1}
Assume $\alpha_i\in \mathbb{N},~1\leq i\leq N.$ Let $u_t^*$ be a sequence of blow up solutions of (\ref{0.3}) with $\rho \notin 8\pi \mathbb{N}.$ Then $u_t^*$ blows up only at $\mathfrak{q}\in M$. Furthermore, there exist a function $w^*$ such that
$$u_t^*\to w^*~\mathrm{in}~C_{\mathrm{loc}}^2(M\setminus\{\mathfrak{q}\})$$
as $t\to0$, and $w^*$ satisfies:
\begin{equation}
\label{0.8}
\Delta w^*+(\rho-8m\pi)\left(\frac{h^*e^{w^*}}{\int_Mh^*e^{w^*}}-1\right)
=4\pi\left(\sum_{i=1}^d\alpha_i-2m\right)(\delta_q-1) + 4\pi\sum\limits_{i=d+1}^N \alpha_i (\delta_{q_i}-1),
\end{equation}
for some $m\in\mathbb{N}$ with $1\leq m\leq[\frac12\sum_{i=1}^d\alpha_i]$ \footnote{$[x]$ stands for the integer part of $x$.} and $\rho > 8m\pi $.
Moreover, $$\lim_{t\to0} \int_{M} h^*e^{u^*_t} = \frac{\rho}{\rho - 8m\pi}\int_{M} h^*e^{w^*}.$$
\end{theorem}

It is clear that the statement of Theorem \ref{th1.1} should be understood along sequences, namely with $t=t_n$ and $t_n\to0$ as $n\to+\infty.$ However there are situations where actually the statement of Theorem \ref{th1.1} holds for $t\to0.$ For example, when
\begin{align*}
&\rho\in(8\pi l,8\pi(l+1))~\mathrm{and}~\sum_{i=1}^d\alpha_i=2l,~\mbox{for some}~l\in\mathbb{N},\\
&~\quad h^*\equiv 1,~M=\mathbb{S}^2~\mathrm{and}~ \alpha_i=0,~\forall i=d+1,\cdots,N,
\end{align*}
then, by virtue of the necessary condition (see (3.25) of  \cite{tt}) for the solvability of problem \eqref{0.8} under the above assumptions, we see that (along any sequence) we must have,
\begin{align*}
m=l\quad\mathrm{and}\quad w^*=0.
\end{align*}
Secondly we observe that in the "geometrical" case,
\begin{align*}
\rho=4\pi(2+\sum_{i=1}^N\alpha_i),
\end{align*}
(where a solution of (\ref{0.3}) yields to the conformal factor for a metric on $M=\mathbb{S}^2$ with constant curvature $h^*\equiv1$ away from the conical singularities $q_i$ with angle $2\pi(\alpha_i-1),~\forall i=1,\cdots,N,$) we know that when $\alpha_i\in\mathbb{N}$ then the condition $\sum_{i=1}^N\alpha_i\in2\mathbb{N}$ is necessary and sufficient for the solvability of (\ref{0.3}). Therefore, we may use the "geometrical" probem as a guiding example in the investigation of the blow-up behaviour for solutions of \eqref{0.3} when $\rho\in 8\mathbb{N}.$\\

From Theorem \ref{th1.1}, we can describe the behaviour of $u_t^*$ away from $\mathfrak{q}.$ So we are left to understand its behaviour near $\mathfrak{q}$. For simplicity, we focus to the case where the collapsing vortices are only {\em two}, namely $d=2,$ with  $\alpha_i\in\mathbb{N}$ and $|\alpha_1-\alpha_2|\leq1$.

Let $G(x,p)$ denote the Green's function for the Laplace Beltrami operator $\Delta$ on $M$, that is
\begin{equation*}
\Delta G(x,p)+(\delta_p-1)=0,~\quad\int_MG(x,p)d\sigma(x)=0.
\end{equation*}
We denote the regular part of the Green's function by $R(x,p)$, then
$$G(x,p)=-\frac{1}{2\pi}\log|x-p|+R(x,p).$$
Set
\begin{align}
\label{def-ut-function}
u_t=u_t^*+4\pi\sum_{i=1}^{2}\alpha_iG(x,q_i(t))+4\pi\sum_{i=3}^N\alpha_iG(x,q_i),
\end{align}
and
\begin{align}
\label{def-w-function}
w=w^*-4\pi(2m-\alpha_1-\alpha_2)G(x,\mathfrak{q})+4\pi\sum_{i=3}^N\alpha_iG(x,q_i),
\end{align}
where $w^*$ is the limit of $u_t^*$ in Theorem \ref{th1.1}. We can rewrite equation (\ref{0.3}) as follows
\begin{equation}
\label{equation-ut}
\Delta u_t+\rho\left(\frac{he^{u_t-4\pi\sum_{i=1}^{2}\alpha_iG(x,q_i(t))}}
{\int_Mhe^{u_t-4\pi\sum_{i=1}^{2}\alpha_iG(x,q_i(t))}}-1\right)=0
\end{equation}
with $h=h^*\exp(-4\pi\sum_{i=3}^N\alpha_iG(x,q_i))$. Thus, $h(x)$ is a non-negative $C^2$ function with only finitely many zeroes (at $q_i$ with multiplicity $\alpha_i$ for $i=3,\cdots,N$), while $h(x)>0$ in a neighborhood of $\mathfrak{q}.$ From Theorem \ref{th1.1} we have that $u_t\to w+8m\pi G(x,\mathfrak{q})$ in $C_{\mathrm{loc}}^2(M\setminus\{\mathfrak{q}\})$ and $w$ satisfies
\begin{equation}
\label{1.5}
\Delta w+(\rho-8m\pi)\left(\frac{he^{w+4\pi(2m-\alpha_1-\alpha_2)G(x,\mathfrak{q})}}
{\int_Mhe^{w+4\pi(2m-\alpha_1-\alpha_2)G(x,\mathfrak{q})}}-1\right)=0.
\end{equation}

Next, we want to investigate the behaviour of $u_t$ in a neighborhood of $\mathfrak{q}.$ By introducing isothermal coordinates around $\mathfrak{q}$, we assume that,
\begin{equation*}
\mathfrak{q}=0,~q_1(t)=t\underline{e}~\mathrm{and}~q_2(t)=-t\underline{e}~\mathrm{for~some}~\underline{e}\in\mathbb{S}^1.
\end{equation*}
Let
\begin{equation}
\label{def-gt}
\begin{aligned}
G_t^{(2)}(x):&=4\pi\alpha_1G(x,q_1(t))+4\pi\alpha_2G(x,q_2(t))\\
&=-2\alpha_1\log|x-t\underline{e}|-2\alpha_2\log|x+t\underline{e}|+R_t,\\
R_t(x):&=4\pi\alpha_1R(x,q_1(t))+4\pi\alpha_2R(x,q_2(t)),
\end{aligned}
\end{equation}
and let $\psi(x)$ satisfy the local problem:
\begin{equation}
\label{def-psi}
\Delta\psi-\rho=0~\mathrm{in}~B_1,~\psi(0)=\nabla\psi(0)=0.
\end{equation}
Therefore we can formulate the local version of (\ref{equation-ut}) around $\mathfrak{q}$ as follows:
\begin{equation}
\label{eq115}
\Delta\bar u_t + h_1(x)|x-t\underline{e}|^{2\alpha_1}|x+t\underline{e}|^{2\alpha_2} e^{\bar u_t} = 0\ \text{in}\ B_1(0),
\end{equation}
where
\begin{equation}
\label{eq116}
\bar u_t=u_t-\log\int_Mhe^{u_t-G_t^{(2)}}-\psi,%-4\pi\alpha_1G(x,q_1(t))-4\pi\alpha_2G(x,q_2(t))
\end{equation}
and
\begin{equation}
\label{def-h1}
h_1(x)=\rho h(x)e^{\psi-R_t},\quad h_1(x)>0~\mathrm{in}~B_1(0).
\end{equation}
Furthermore, without loss of generality, we can turn the "global" character of $u_t$ over $M$ into the following "local" information about $\bar u_t:$
\begin{equation}
\label{assume-item}
\left\{\begin{array}{l}
\bar u_t~\mbox{admits {\em only} the origin as a blow-up point in}~\bar B_1(0),\\
|\bar u_t(x)-\bar u_t(y)|\leq C,~\forall x, y\in \partial B_1(0),\\
\int_{B_1(0)} h_{1,t}(x)e^{\bar u_t}dx \leq C,
\end{array}\right.
\end{equation}
where
\begin{equation}
\label{def-h1t}
h_{1,t}(x) = h_1(x)|x-t\underline{e}|^{2\alpha_1} |x+t\underline{e}|^{2\alpha_2}.
\end{equation}
\medskip

In order to study the behaviour of $\bar u_t$ near the origin, we consider the scaled sequence
\begin{equation}
\label{1-vt}
v_t(y) =\bar u_t(ty) + 2(1+\alpha_1+\alpha_2)\text{log}\ t,\ x\in B_{1/t},
\end{equation}
which satisfies:
\begin{equation}
\label{eq117}
\begin{cases}
\Delta v_t + h_t(y) e^{v_t} = 0\ \ \text{in}\ B_{1/t},\\
\int_{B_{1/t}} h_t(y) e^{v_t} \leq C,
\end{cases}
\end{equation}
with
\begin{equation}
\label{eq118}
h_t(y) = h_1(ty)|y-\underline{e}|^{2\alpha_1} |y+\underline{e}|^{2\alpha_2}.
\end{equation}
Let $\sigma_u$ and $m_v$ denote the local mass of $\bar u_t$ in the two different scales:
\begin{subequations}
\begin{align}
&\sigma_u =\frac{1}{2\pi}\lim_{r \rightarrow 0}\lim_{t \rightarrow 0} \int_{B_{r}(0)} h_{1,t}(x) e^{\bar u_t(x)} dx, \ \quad \\
&m_v =\frac{1}{2\pi}\lim_{R \rightarrow +\infty}\lim_{t \rightarrow 0} \int_{B_{R}(0)} h_t(x) e^{v_t(x)} dx.
\end{align}
\end{subequations}
We prove:
\begin{theorem}
\label{th1.2}
Suppose that $\rho\notin 8\pi\mathbb{N},$ $h_1(x)>0~\mathrm{in}~B_1(0),~ \alpha_1 ,~ \alpha_2 \in\mathbb{N} ~with~ |\alpha_1-\alpha_2|\leq1$ and $\alpha_i\in\mathbb{N}\cup\{0\}~for ~3\leq i\leq N.$ Let $\bar u_t$ satisfy (\ref{eq115}), (\ref{assume-item}), then as $t\to0,$ $v_t$ in (\ref{1-vt}) blows up at finitely many (distinct) points:
\begin{equation*}
\{p_1, \cdots ,p_m \} \subset \mathbb{R}^2 \backslash \{ \underline{e}\ , -\underline{e} \}
\end{equation*}
with $1 \leq m \leq \min\{ \alpha_1 ,\alpha_2 \}$. Furthermore,
\begin{equation*}
\sigma_u = m_v=4m.
\end{equation*}
\end{theorem}

We emphasize that in Theorem \ref{th1.1} and Theorem \ref{th1.2}, the parameter $t$ can be used to control the blow up behavior of both $u_t$ and $v_t$ as long as $\rho\not\in 8\pi\mathbb{N}.$ Next, we fix $r_0 > 0$ sufficiently small and $R > 1$ sufficiently large such that,
\begin{equation*}
B_{4r_0}(p_i) \cap B_{4r_0}(p_j) = \emptyset\ \text{for}\ i \neq j,
\end{equation*}
and
\begin{equation}
\label{eq120}
\bigcup_{i=1}^m B_{4r_0}(p_i) \subset B_{R}(0).
\end{equation}
Furthermore, we let:
\begin{subequations}
\begin{align}
&\lambda_{t,i} = \max\limits_{\overline{B}_{r_0}(p_i)} v_t = v_t(p_{t,i}),\quad \ p_{t,i} \rightarrow p_i\ \text{as}\ t \rightarrow 0\\
&i = 1, \cdots , m,\ \text{and}\ \lambda_t = \max\{ \lambda_{t,i},\ i=1,\cdots , m \}.
\end{align}
\end{subequations}
We have:
\begin{theorem}
\label{th1.3}
Under the assumptions of Theorem 1.2  we have:
\begin{enumerate}
\item[(i)] $|\lambda_t + 2(1+\alpha_1 + \alpha_2 -2m)\log t | \leq C,$
\item[(ii)] $||\exp\left(u_t-G_t^{(2)}\right)||_{L^\infty(M \backslash B_{\lambda t}(\mathfrak{q}))} \leq C,$
\end{enumerate}
with $\lambda > 0$ sufficiently large and $C$ a suitable positive constant.
\end{theorem}

The blow up behaviour of $v_t$ around each blow up point has been studied in \cite{l} and \cite{cl1} via {\em a second time re-scaling} (from the original $u_t$), which is necessary to obtain an accurate behaviour of $v_t$ in a neighbourhood of each blow up point.

Next, we would like to give refined estimates than those provided in Theorem 1.3. Obviously, we need to consider (\ref{equation-ut}) globally in order to achieve this goal. From Theorem 1.1, we know that: $u_t\rightarrow w+8m\pi G(x,\mathfrak{q})$ in $C_{\mathrm{loc}}^2(M\setminus\{\mathfrak{q}\})$, where $w$ satisfies (\ref{1.5}). To get refined estimate between $\lambda_{t,i}$ and $\log t$, we need to assume that $w$ is a non-degenerate solution to (\ref{1.5}).
\medskip

\noindent {\bf Definition 1.1.} A solution $w$ of (\ref{1.5}) is said non-degenerate, if the linearized problem
\begin{align}
\label{1.6}
\Delta \phi+(\rho-8m\pi)\frac{he^{w_m}}{\int_Mhe^{w_m}}
\left(\phi-\frac{\int_Mhe^{w_m}\phi}{\int_Mhe^{w_m}}\right)=0,\quad \int_M\phi=0
\end{align}
only admits the trivial solution. Here $w_m=w+4\pi(2m-\alpha_1-\alpha_2)G(x,\mathfrak{q}).$
\medskip

Using the transversality theorem, we can always choose a non-negative smooth function $h$ such that $w$ is non-degenerate. See Theorem 4.1 in \cite{llwy}.

Based on the non-degeneracy assumption for (\ref{1.5}), we can obtain sharper estimates on the bubbling solutions of (\ref{equation-ut}). To state our result, we introduce the following notations:
\begin{align}
\label{1.7}
\rho_{t,i}=\int_{B_{r_0 t}(tp_{t,i})} h_{1,t} e^{\bar u_t} = \frac{\int_{U_{i,t}^{r_0}} \rho h^*e^{u^*_t}}{\int_M h^*e^{u^*_t}},
\end{align}
\begin{align}
\label{1.8}
C_{t,i}=\frac{1}{8}h_1(tp_{t,i})|p_{t,i}-e|^{2\alpha_1}|p_{t,i}+e|^{2\alpha_2},
\end{align}
where $U_{i,t}^{r_0}$ is the corresponding neighborhood in $M$ of the ball $B_{tr_0}(tp_{t,i})$ under the isothermal coordinate. We denote by $y_{i,t}$ the point of $M$ which corresponds to $tp_{t,i}$ in the isothermal coordinate. Throughout the paper, without causing confusion, we might identify $y_{i,t},~U_{i,t}^{r_0},~U_{i,t}^{2r_0}$ with $tp_{t,i},~B_{r_0 t}(tp_{t,i}),~B_{2r_0 t}(tp_{t,i})$ respectively. Then we have the following:

\begin{theorem}
\label{th1.4}
Let $u_t$ be the sequence of blow up solutions of (\ref{equation-ut}) and $w+8m\pi G(x,\mathfrak{q})$ be its limit in $M\setminus\{\mathfrak{q}\}.$ If $w$ is a non-degenerate solution of (\ref{1.5}), then under the assumptions of Theorem \ref{th1.2}  we have:
\begin{itemize}
  \item \begin{equation}
        \label{1.10}
        \|u_t-w-\sum_{i=1}^m\rho_{t,i}G(x,y_{i,t})\|_{C^1(M\setminus\bigcup_{i=1}^m U_{i,t}^{2r_0})}\leq Ct\log t,
        \end{equation}
  \item \begin{equation}
        \label{1.11}
        \begin{aligned}
       ~~ &\lambda_{t,i}+(2+2\alpha_1+2\alpha_2-4m)\log t-\log\left(\frac{\rho}{\rho-8m\pi}\int_Mhe^{w_m}\right)+w(tp_{t,i})\\
        &~+2\log C_{t,i}-\sum_{j\neq i}4\log|p_{t,i}-p_{t,j}|+\sum_{j=1}^m8\pi R(y_{i,t},y_{j,t})=O(t\log t),
        \end{aligned}
        \end{equation}
  \item $\rho_{t,i}-8\pi=O(t^2\log t)$ and
        \begin{align}
        \label{1.12}
        \left|\int_Mhe^{u_t-G_t^{(2)}}-\frac{\rho}{\rho-8m\pi}\int_Mhe^{w_m}\right|\leq Ct.
        \end{align}
\end{itemize}
\end{theorem}

For the blow up solution $u_t$, another interesting issue is to describe the position of the blow up points of the corresponding sequence $v_t$ in (\ref{1-vt}). They are determined by the following $m$-identities:
\begin{equation}
\label{1.13}
\alpha_1\frac{p_i-\underline{e}}{|p_i-\underline{e}|^2}+\alpha_2\frac{p_i+\underline{e}}{|p_i+\underline{e}|^2}=2\sum_{j\neq i}\frac{p_i-p_j}{|p_i-p_j|^2},~1\leq i\leq m.
\end{equation}
In section 6 we shall prove that the above equations (\ref{1.13}) are uniquely solvable. In other words, the blow up points of $v_t$ are uniquely determined. This fact will be very important when showing uniqueness of the blow up solutions $u_t$.

In concluding, we want to say a few words about the purpose of the results we have obtained. Obviously, the blow up phenomenon of collapsing singularities is interesting by itself. On a more important side, it arises naturally in the computation for the degree formula of the following shadow system, see \cite[Theorem 1.4]{llwz},
\begin{align}
\label{1.14}
\left\{\begin{array}{l}
\Delta \mathrm{w}+\rho\left(\frac{h_2e^{\mathrm{w}-4\pi G(x,Q)}}{\int_Mh_2e^{\mathrm{w}-4\pi G(x,Q)}}-1\right)=0,\\
\nabla(\log h_1e^{-\frac12a\mathrm{w}})\mid_{x\in Q}=0,
\end{array}\right.
\end{align}
where $a\in\{1,2,3\}$, $h_i(x)=h_i^*(x)e^{4\pi\sum_{p\in S_i}\alpha_{p,i}G(x,p)},$ $h_i^*$ is a smooth positive functions on $M,$ $\alpha_{p,i}\in\mathbb{N}$, $S_i$ is a finite set in $M$, $i=1,2$ and $Q\notin S_1.$ To prove a priori bounds for solutions to (\ref{1.14}), it is unavoidable to face the difficulty of collapsing singularities. Indeed, for a sequence of solutions $(\mathrm{w}_k,Q_k)$ of (\ref{1.14}) with $Q_k\notin S_1\cup S_2$, it might happen that $Q_k\to Q_0\in S_2$. For details, we refer the readers to \cite{llwy,llwz}. Our analysis here aims to contribute to clarify this situation.
\medskip

This paper is organised as follows. In section 2 we prove Theorem B.2 and in section 3 we prove Theorem \ref{th1.1}. In section 4, we study the behaviour of the blow up solution $v_t$ and give the proof to Theorem \ref{th1.2} there. For the uniqueness of the blow up points of $v_t$, we provide the proof in section 6. Lastly, in Appendix A we provide some technical estimates, while in Appendix B we discuss the solvability of problem (\ref {new13}) below,  which is of independent interest.

\section{The proof of Theorem B.2}
It follows from the results in \cite{cl2}, that for every $\rho \in (8\pi,16\pi)$ and for $h_k \in C^1(\mathbb{S}^2):\int_{\mathbb{S}^2}h_k d\sigma = 1$, the problem:
\begin{equation}
\label{new1}
\Delta u_k^* + \rho(\frac{e^{u_k^*}}{\int_{\mathbb{S}^2}e^{u_k^*}}-\frac{1}{4\pi}) = 4\pi\alpha(h_k-\frac{1}{4\pi})\ \ \text{in}\ \mathbb{S}^2
\end{equation}
admits at least a solution $u_k^* \in \mathring{H}^1(S^2).$ Furthermore, without loss of generality, after a rotation we can assume that the point $p \in \mathbb{S}^2$ in Theorem B.2 coincided with the south pole of $\mathbb{S}^2$ located at $(0,0,-1)$. Hence, the north pole of $\mathbb{S}^2$ is located at $(0,0,1)$ and from it we consider the stereographic projection, $\pai: \mathbb{S}^2 \rightarrow \mathbb{R}^2$. By letting:
\begin{equation}
\label{new2}
v_k(x)=u^*_k(y)-\text{log}\int_{\mathbb{S}^2}e^{u^*_k} d\sigma+\text{log}(\frac{4\rho}{(1+|x|^2)^{a+2}})
\end{equation}
with $x =\pai(y)$, $y\in\mathbb{S}^2$, we see that $v_k$ satisfies:
\begin{equation}
\label{new3}
\begin{cases}
\Delta v_k + (1 + |x|^2)^{a} e^{v_k} = 4\pi\alpha g_k(x)\ \ \text{in}\ \mathbb{R}^2,\\
\int_{\mathbb{R}^2}(1+|x|^2)^{a}e^{v_k} dx = \rho,
\end{cases}
\end{equation}
with
\begin{equation}
\label{new4}
a=\frac{\rho}{4\pi}-(\alpha+2)
\end{equation}
and
\begin{equation}
\label{new5}
g_k(x)=\frac{4h_k(\pai^{-1}(x))}{(1+|x|^2)^2}.
\end{equation}
In particular, we have $\int_{\mathbb{R}^2}g_k(x)dx = 1$.

We take $g_k$ as the standard regularisation of $\delta_0$ namely,
\begin{equation}
\label{new6}
g_k(x)=\frac{\lambda_k}{\pi(1+\lambda_k|x|^2)^2},\ \ \lambda_k \rightarrow +\infty\ \text{as}\ k\rightarrow+\infty.
\end{equation}
So,
\begin{equation}
\label{new7}
g_k \rightharpoonup \delta_0\ \  \text{weakly in the sense of measure in}\ \mathbb{R}^2;
\end{equation}
and in turn its pull-back $h_k$ over $\mathbb{S}^2$ (see (\ref{new5})) also satisfies:
\begin{equation}
\label{new8}
h_k \rightarrow \delta_p\ \ \text{weakly in the sense of measure on}\ \mathbb{S}^2.
\end{equation}

To proceed further, we recall some useful facts. First of all, it was shown in \cite{tt} by means of a Pohozaev type identity that the problem:
\begin{equation}
\label{new9}
\begin{cases}
\Delta v + (1 + |x|^2)^{a} e^{v} = 4\pi\alpha \delta_0\ \ \text{in}\ \mathbb{R}^2,\\
\int_{\mathbb{R}^2} (1 + |x|^2)^{a} e^{v} dx = \rho,
\end{cases}
\end{equation}
with
\begin{equation}\label{new10}
0 \neq \alpha > -1\ \ \text{and $a$ satisfying}\ \ (\ref{new4})
\end{equation}
is solvable only if
\begin{equation}\label{new11}
\rho \in (0,8\pi(1-\alpha^-)) \cup (8\pi(1+\alpha^+),+\infty),
\end{equation}
where, as usual, $\alpha^{\pm}=\ \text{max}\{ 0,\pm\alpha \}$.

By means of the transformation:
\begin{equation*}
V(x) =v(\frac{x}{|x|^2})+(\frac{\rho}{2\pi}-2\alpha)\log\frac{1}{|x|},
\end{equation*}
we check easily that $V$ extends smoothly at the origin and it satisfies:
\begin{equation}\label{new12}
\begin{cases}
\Delta V + (1 + |x|^2)^{a} e^{V} = 0,\\
\int_{\mathbb{R}^2} (1 + |x|^2)^{a} e^{V} dx = \rho,
\end{cases}
\end{equation}
with $a=\frac{\rho}{4\pi}-(\alpha+2)$ and $0 \neq \alpha > -1$. So that (\ref{new11}) is also a necessary condition for the solvability of (\ref{new12}).

In addition to carry out our blow-up analysis and in order to motivate the assumption (\ref{0.33}), we consider also the problem:
\begin{equation}
\label{new13}
\begin{cases}
\Delta V + (1 + |x|^2)^{\alpha} e^{V} = 0\ \ \text{in}\ \mathbb{R}^2,\\
\int_{\mathbb{R}^2} (1 + |x|^2)^{\alpha} e^{V} dx = \rho.
\end{cases}
\end{equation}
Notice that now there is no relation (of the type (\ref{new4})) which links the power $\alpha$ to $\rho$. Nonetheless we show that when $\rho$ satisfies (\ref{0.33}) then problem (\ref{new13}) admits \textit{no} solutions.

More precisely, for problem (\ref{new13}) the following holds:
\begin{proposition}
\label{pr2.new}
Let $\alpha>-1.$
\begin{itemize}
  \item[(i)] If $\rho\in(0,16\pi)$ then a solution of (\ref{new13}) is necessarily radially symmetric.
  \item[(ii)] If $-1<\alpha\leq1$ then problem (\ref{new13}) is solvable if and only if $$\rho\in(8\pi\min\{1,\alpha+1\},8\pi\max\{1,\alpha+1\}).$$
  In addition, the corresponding solution is unique, radially symmetric and non-degenerate.
  \item[(iii)] If $\alpha>1$ then there exists $\bar\rho_\alpha\in(4\pi(\alpha+1),8\pi\alpha)$ such that problem (\ref{new13}) admits a radially symmetric solution if and only if $\rho\in[\bar\rho_\alpha,8\pi(\alpha+1)).$ Furthermore, if $\rho\in[8\pi\alpha,8\pi(\alpha+1))$ then the corresponding radial solution is unique and non-degenerate. While for $\rho\in(\bar\rho_{\alpha},8\pi\alpha)$ problem (\ref{new13}) admits at least two radial solutions.
\end{itemize}
\end{proposition}

We postpone the proof of Proposition \ref{pr2.new} in Appendix B.
\medskip

\noindent {\bf Remark 2.1:} We note that, if $1<\alpha<3$  and we let,
\begin{equation}
\label{2.new.0}
\rho_{\alpha}^*=\min\{\bar\rho_\alpha,16\pi\}
\end{equation}
then for $\rho\in(4\pi(\alpha+1),\rho_{\alpha}^*)$ problem (\ref{new13}) admits no solution.  We suspect that such non-existence property actually holds for problem (\ref{new13})  whenever $\rho\in(4\pi(\alpha+1),\bar\rho_{\alpha})$ and $\alpha>1.$
\medskip

At this point we pass to analyse the sequence $v_k$ satisfying (\ref{new3}) and (\ref{new6}). We know that $v_k$ must blow up, in view of (\ref{new7}) and (\ref{new11}). To describe its asymptotic behaviour as $k \rightarrow +\infty$, we consider the new sequence:
\begin{equation*}
\xi_k(x)=v_k(x)-\alpha\ \text{log}(\frac{1}{\lambda_k}+|x|^2).
\end{equation*}
We can easily check that,
\begin{equation}
\label{new14}
\begin{cases}
\Delta \xi_k+(\frac{1}{\lambda_k} + |x|^2)^{\alpha}(1+|x|^2)^{a} e^{\xi_k}=0,\ \ \text{in}\ \mathbb{R}^2,\\
\int_{\mathbb{R}^2}(\frac{1}{\lambda_k}+|x|^2)^{\alpha}(1+|x|^2)^{a} e^{\xi_k} = \rho.
\end{cases}
\end{equation}
(Recall that $g_k$ is given by (\ref{new6}) with $a$ satisfying (\ref{new4}).)

Notice that the blow-up analysis of Brezis-Merle \cite{bm} and Li-Shafrir \cite{ls} does not apply to $\xi_k$ near the origin, and in fact the following holds:
\medskip

\begin{theorem}
\label{newth2.1}
Assume (\ref{new4}) and (\ref{0.33}) and let $\xi_k$ satisfy (\ref{new14}) then (along a subsequence) the following holds:
\begin{equation*}
\begin{aligned}
\xi_k(x_k)=\max_{\mathbb{R}^2}\xi_k\to\infty~\mathrm{and}~x_k\to0,~\mathrm{as}~k\to+\infty,
~\xi_k\to\xi~\mathrm{in}~C_{\mathrm{loc}}^2(\mathbb{R}^2\setminus\{0\}),
\end{aligned}
\end{equation*}
with $\xi$ satisfying:
\begin{equation*}
\begin{cases}
\Delta \xi + |x|^{2\alpha}(1+|x|^2)^{a}e^{\xi}+ 8\pi\delta_0 = 0,\\
\int_{\mathbb{R}^2}|x|^{2\alpha}(1+|x|^2)^{a}e^{\xi} = \rho-8\pi.
\end{cases}
\end{equation*}
\end{theorem}

\begin{proof}
We start by observing that the function:
\begin{equation*}
\hat{\xi}_{k}(x) = \xi_k (\frac{x}{|x|^2}) - \frac{\rho}{2\pi}\text{log}|x|
\end{equation*}
extends smoothly at the origin and satisfies:
\begin{equation*}
\begin{cases}
-\Delta \hat{\xi}_{k} = (\frac{|x|^2}{\lambda_k} + 1)^{\alpha}(1+|x|^2)^{a} e^{\hat{\xi}_{k}}\ \ \text{in}\ \mathbb{R}^2,\\
\int_{\mathbb{R}^2}(\frac{|x|^2}{\lambda_k}+1)^{\alpha}(1+|x|^2)^{a} e^{\hat{\xi}_{k}} = \rho,
\end{cases}
\end{equation*}
and the well known blow-up analysis of \cite{bm,ls} applies to $\widehat{\xi_k}$.

\begin{flushleft}
{\bf Claim 1:} The origin is a blow-up point for $\xi_k$.
\end{flushleft}

Indeed, if by contraction this was not the case, then $\hat{\xi}_{k}$ would either be uniformly bounded (locally in $\mathbb{R}^2$) or it would blow-up in $\mathbb{R}^2$. Notice that the alternative: $\max_{B_R}\hat{\xi}_{k} \rightarrow -\infty$ as $k \rightarrow +\infty$, $\forall\ R > 0$ cannot occur. Indeed, since (by contradiction) the origin is not a blow up point for $\xi_k$, then it would imply (by a standard Harnack-type inequality) that also $\max_{B_R} \xi_k \to -\infty$ as $k \rightarrow +\infty$, $\forall\ R > 0$, and consequently we  would get  $\rho=0$, which is impossible. Next we rule out the possibility that $\hat{\xi}_{k}$ is (locally) uniformly bounded. In fact it would imply that (along a subsequence), $\hat{\xi}_{k}\rightarrow v$ with $v$ satisfying (\ref{new12}). This is impossible since our assumption on $\rho$ violates the necessary condition (\ref{new11}). Finally if $\hat{\xi}_{k}$ blows-up, then by applying \cite{bm,ls} and by using the contradiction assumption (i.e. the origin is not a blow up point for $\xi_k$), we would derive that necessarily $\rho=8\pi,$ in contradiction to our assumptions on $\rho$ in (\ref{0.33}).  In all cases we have obtained a contradiction and Claim 1 is established.

We let
\begin{equation}
\label{new214a}
m_0 = \lim_{\varepsilon \rightarrow 0} \lim\limits_{k \rightarrow +\infty} \int_{|x| < \varepsilon} (\frac{1}{\lambda_k}+|x|^2)^{\alpha}(1+|x|^2)^a e^{\xi_k}
\end{equation}
be the local blow-up mass of $\xi_k$ at the origin. By means of a Pohozaev type identity applied in the usual way (see \cite{tt}) we find that,
\begin{equation*}
\frac{m_0}{4\pi}(m_0-8\pi)=2\alpha\lim_{r\to0}\lim_{k\to+\infty} \int_{|x| \leq r}(\frac{1}{\lambda_k}+|x|^2)^{\alpha-1}|x|^2(1+|x|^2)^a e^{\xi_k} \geq0,
\end{equation*}
and so,
\begin{equation}
\label{new214b}
m_0 \geq 8\pi.
\end{equation}
By our assumption on $\rho$ and (\ref{new214b}), we see that necessarily $\hat{\xi}_{k}$ cannot blow-up. Therefore we conclude that the origin is the only blow-up point for $\xi_k$, and $\forall\ \varepsilon >0$ there exists a constant $C_{\varepsilon} >0$
\begin{equation}
\label{new15}
\xi_k(x)+\frac{\rho}{2\pi}\text{log}|x| \leq C_{\varepsilon},\ \ \forall\ |x| \geq \varepsilon.
\end{equation}
We show that the estimate (\ref{new15}) extends away of a tiny neighborhood of the origin with size $\frac{1}{\lambda_k}$. To this purpose we introduce the scaled sequence:
\begin{equation*}
\phi_k(x) = \xi_k(\frac{x}{\sqrt{\lambda_k}}) + (\alpha+1)\text{log}\frac{1}{\lambda_k}
\end{equation*}
which satisfies:
\begin{equation}
\label{new16}
\begin{cases}
\Delta\phi_k+(1+|x|^2)^{\alpha}(1+\frac{|x|^2}{\lambda_k})^{a}e^{\phi_k} = 0\ \ \text{in}\ \mathbb{R}^2,\\
\int_{\mathbb{R}^2}(1+|x|^2)^{\alpha}(1+\frac{|x|^2}{\lambda_k})^{a}e^{\phi_k} = \rho.
\end{cases}
\end{equation}

\begin{flushleft}
{\bf Claim 2:} $\phi_k$ must blow up.
\end{flushleft}

To establish the above claim, we observe again that the standard blow-up analysis of \cite{bm,ls} applies to $\phi_k$. Therefore if by contradiction we assume that $\phi_k$ does not blow-up, then either $\phi_k$ is locally uniformly bounded or,
\begin{equation}
\label{new18}
\begin{aligned}
\sup_{B_R} \phi_k \rightarrow - \infty\ \ \text{as}\ k \rightarrow +\infty,\ \ \forall\ R>0.\\
\end{aligned}
\end{equation}
We readily rule out the possibility that $\phi_k$ is uniformly bounded. Indeed, if this was the case, then along a subsequence we would find that $\phi_k \to v$ in $C^2_{\mathrm{loc}}(\mathbb{R}^2)$, with $v$ satisfying (\ref{new13}). But we know that this is not possible, since under the assumption (\ref{0.33}) problem (\ref{new13}) admits no solutions, see Remark 2.1.

In order to see that also (\ref{new18}) is not allowed, we consider,
\begin{equation}
\label{new19}
\hat{\phi}_k(x) =\phi_k(\frac{x}{|x|^2})-\frac{\rho}{2\pi}\text{log}|x|
\end{equation}
satisfying:
\begin{equation}
\label{new20}
\begin{cases}
\Delta \hat{\phi}_{k} + (1+|x|^2)^{\alpha}(\frac{1}{\lambda_k}+|x|^2)^{a}e^{\hat{\phi}_{k}} = 0,\\
\int_{\mathbb{R}^2}(1+|x|^2)^{\alpha}(\frac{1}{\lambda_k}+|x|^2)^{a}e^{\hat{\phi}_{k}} = \rho.
\end{cases}
\end{equation}
In view of (\ref{new18}), we see that $\hat{\phi}_{k}$ must blow up at the origin and concentration must occur. In other words, (along a subsequence)
\begin{equation*}
(1+|x|^2)^{\alpha}(|x|^2+\frac{1}{\lambda_k})^{a}e^{\hat{\phi}_{k}} \rightarrow \rho \delta_0,\ \ \text{as}\ k \rightarrow +\infty.
\end{equation*}
Observe that under the given assumption (\ref{0.33}) on $\rho$ and (\ref{new4}), we see that necessarily,
\begin{equation}
\label{new21}
-1<a<0.
\end{equation}
Thus, by using a Pohozaev type inequality as above, in this case we obtain:
\begin{equation*}
\frac{\rho}{4\pi}(\rho-8\pi)=
\lim_{r\to0}\lim_{k\to+\infty}2a\int_{|x|\leq r}(1+|x|^2)^{\alpha}(|x|^2+\frac{1}{\lambda_k})^{a-1}|x|^2 e^{\hat{\phi}_{k}} \leq 0,
\end{equation*}
As a consequence, $\rho \leq 8\pi$, in contradiction with (\ref{0.33}), and Claim 2 is established.

So, we can use \cite{bm,ls} and by (\ref{0.33}), we conclude that $\phi_k$ must blow-up exactly at one point, say $q \in \mathbb{R}^2$, and as $k \rightarrow +\infty$ (along a subsequence)
\begin{equation}
\label{new22}
(1+|x|^2)^{\alpha}(1+\frac{|x|^2}{\lambda_k})^{a}e^{\phi_k} \rightharpoonup 8\pi \delta_q
\end{equation}
weakly in the sense of measure (locally) in $\mathbb{R}^2$.

In particular, for any $R>0$ sufficiently large, we have:
\begin{equation}
\label{new22a}
\int_{|x| \leq \frac{R}{\sqrt{\lambda_k}}}(\frac{1}{\lambda_k}+|x|^2)^{\alpha}(1+|x|^2)^{a}e^{\xi_k}=
\int_{|x| \leq R}(1+|x|^2)^{\alpha}(1+\frac{|x|^2}{\lambda_k})^{a}e^{\phi_k} \rightarrow 8\pi\ \ \mbox{as}\ k \rightarrow +\infty.
\end{equation}
As a consequence of (\ref{new22}) and (\ref{new22a}) and in view of our assumption on $\rho$ in (\ref{0.33}), we see that necessarily also the sequence $\hat{\phi}_{k}$ in (\ref{new19}) must blow-up at the origin and "concentration" must occur. In other words, (along a subsequence) for $r_0 > 0$ sufficiently small, as $k \to+\infty$, we have:
\begin{equation}
\label{new23}
(1+|x|^2)^{\alpha}(\frac{1}{\lambda_k}+|x|^2)^{a}e^{\hat{\phi}_{k}} \rightharpoonup \beta_{\infty} \delta_0,
\end{equation}
weakly in the sense of measure in $B_{r_0}$, with
\begin{equation}
\label{new24}
\beta_{\infty} =\lim_{r\to 0^+}\lim_{k \to +\infty}\int_{B_r}(1+|x|^2)^{\alpha}(\frac{1}{\lambda_k}+|x|^2)^{a}e^{\hat{\phi}_k(x)}dx.
\end{equation}
Our next goal is to identify the value $\beta_{\infty}$. To this purpose we use a Pohozaev type inequality (in the usual way) to obtain:
\begin{equation}
\label{new25}
\begin{aligned}
\beta_{\infty}(\frac{\beta_{\infty}}{4\pi}-2)
=~&2\alpha\lim_{r\to 0^+}\lim_{k\to+\infty}\int_{|x| < r}(1+|x|^2)^{\alpha-1}(\frac{1}{\lambda_k}+|x|^2)^{a}|x|^2e^{\hat{\phi}_k}dx\\
&+2a\lim_{r\to 0^+}\lim_{k\to+\infty}\int_{|x|<r}(1+|x|^2)^{\alpha}(\frac{1}{\lambda_k}+|x|^2)^{a-1}|x|^2e^{\hat{\phi_k}}dx\\
=~&2a\beta_{\infty}-2a\lim_{r\to0^+}\lim_{k\to+\infty}\int_{|x|<r}\frac{(1+|x|^2)^{\alpha}(\frac{1}{\lambda_k}+|x|^2)^{a}e^{\hat{\phi}_k}}{\lambda_k|x|^2+1}dx,
\end{aligned}
\end{equation}
where we have used (\ref{new23}) in order to see that the first integral in (\ref{new25}) vanishes. Thus, from (\ref{new25}) we find:
\begin{equation}
\label{new26}
\beta_{\infty}(\frac{\beta_{\infty}}{4\pi}-2(a+1))
=-2a\lim_{r\to 0^+}\lim_{k\to+\infty}\int_{|x|<r}\frac{(1+|x|^2)^{\alpha}(\frac{1}{\lambda_k}+|x|^2)^{a}e^{\hat{\phi}_k}}{\lambda_k|x|^2+1}
\end{equation}
and, by recalling (\ref{new21}), we deduce that,
\begin{equation}\label{new27}
\beta_{\infty} \geq 8\pi(a+1).
\end{equation}
In order to estimate the integral in (\ref{new26}), we analyse the blow-up behaviour of $\hat{\phi}_k$ around the origin.

To this purpose, we let $\hat{x}_k \in \overline{B}_{r_0}(0):$
\begin{equation}
\label{new28}
\hat{\phi}_k(\hat x_k)=\max_{|x|\leq r_0}\hat\phi_k\to+\infty,~\mbox{as}~k\to+\infty.
\end{equation}
We check that,
\begin{equation}
\label{new29}
\hat{\phi}_k(\hat{x}_k)+2(a+1)\log(\max\{\frac{1}{\sqrt{\lambda_k}},|\hat{x}_k|\})\leq C
\end{equation}
for suitable $C > 0$. To establish (\ref{new29}), we argue by contradiction and (along a subsequence) we suppose that,
\begin{equation*}
\hat{\phi}_k(\hat x_k)+2(a+1)\log(\max\{\frac{1}{\sqrt{\lambda_k}},|\hat{x}_k|\})\rightarrow+\infty,\ \ \mathrm{as}\ k \rightarrow + \infty.
\end{equation*}
Then by setting:
\begin{equation}
\label{new30}
\tau_k=\max\{\frac{1}{\sqrt{\lambda_k}},|\hat x_k| \} \rightarrow 0,
\end{equation}
as $k \rightarrow +\infty$, and
\begin{equation}
\label{new31}
\psi_k(x)=\hat\phi_k(\hat x_k+\tau_k x)+2(a+1)\log\tau_k,
\end{equation}
we see that $\psi_k$ satisfies:
\begin{equation*}
\begin{cases}
-\Delta \psi_k = (1+|\hat{x}_k+\tau_k x|^{2})^{\alpha}(\frac{1}{\lambda_k \tau_k^2}+|\frac{\hat{x}_k}{\tau_k}+x|^2)^{a}e^{\psi_k(x)} := f_k(x)\ \mbox{in}\ \mathbb{R}^2,\\
\int_{\mathbb{R}^2}f_k(x)dx = \rho,\quad\ \psi_k(0)\rightarrow+\infty~\mbox{as}~k \rightarrow +\infty.
\end{cases}
\end{equation*}
Furthermore, by the definition of $\tau_k$ in (\ref{new30}) we see that, for $\varepsilon >0$ sufficiently small, we have:
\begin{equation*}
1-2\varepsilon \leq(\frac{1}{\lambda_k \tau_k^2}+|\frac{\hat x_k}{\tau_k}+ x|^2)\leq2+3\varepsilon,\ \ \forall\ |x| \leq \varepsilon.
\end{equation*}
So the standard blow-up analysis of \cite{bm,ls} applies to $\psi_k$ around the origin and implies that,
\begin{equation*}
\int_{B_{\varepsilon\tau_k}(\hat{x}_k)}(1+|x|^2)^{\alpha}(\frac{1}{\lambda_k}+|x|^2)^{a}e^{\hat{\phi}_k}
=\int_{|x|<\varepsilon}f_k(x)dx \to 8\pi,\ \ \mbox{as}\ k\to+\infty,
\end{equation*}
which, together with (\ref{new22}) yields to a contradiction of our assumption (\ref{0.33}). Thus (\ref{new29}) is established.

We define,
\begin{equation}
\label{new32}
s_k = e^{-\frac{\hat\phi_k(\hat x_k)}{2(a+1)}}
\end{equation}
and by (\ref{new29}), we find that:
\begin{equation}
\label{new33}
\frac{|\hat{x}_k|}{s_k} \leq C\ \ \mbox{and}\ \ s_k\sqrt{\lambda_k} \geq C,\ \ \forall\ k \in \mathbb{N}
\end{equation}
for a suitable constant $C > 0$. Let
\begin{equation*}
\eta_k(x)=\hat\phi_k(s_k x)-\hat\phi_k(\hat x_k),
\end{equation*}
which satisfies
\begin{equation}
\label{new36}
\begin{cases}
-\Delta\eta_k=(1+s_k^2|x|^2)^{\alpha}(\frac{1}{\lambda_k s^2_k} + |x|^2)^{a}e^{\eta_k(x)}:= \hat f_k(x)\ \ \mbox{in}\ \mathbb{R}^2,\\
\int_{\mathbb{R}^2}\hat f_k(x)dx =\rho,\quad 0=\eta_k(\frac{\hat x_k}{s_k})=\max_{|x|\leq\frac{r_0}{s_k}}\eta_k(x).
\end{cases}
\end{equation}

\begin{flushleft}
{\bf Claim 3:} (along a subsequence)~$\lim_{k\to+\infty}s_k\sqrt{\lambda_k}=L>0.$
\end{flushleft}

We prove it by contradiction. In views of (\ref{new33}), we assume,
\begin{equation}
\label{new35}
s_k\sqrt{\lambda_k} \rightarrow +\infty,\ \ \text{as}\ k \rightarrow +\infty.
\end{equation}
Then, for fixed $r>0$ and $R\gg1$, for large $k\in\mathbb{N}$, we can estimate the integral in (\ref{new26}) as follows:
\begin{equation*}
\begin{aligned}
&\int_{|x|<r}\frac{(1+|x|^2)^{\alpha}(\frac{1}{\lambda_k}+|x|^2)^{a}e^{\hat{\phi}_k}}{\lambda_k|x|^2 + 1}dx\\
=~&\int_{|x|\leq Rs_k}\frac{(1+|x|^2)^{\alpha}(\frac{1}{\lambda_k}+|x|^2)^{a}e^{\hat{\phi}_k}}{\lambda_k|x|^2+1}dx
+\int_{s_kR<|x|<r}\frac{(1+|x|^2)^{\alpha}(\frac{1}{\lambda_k}+|x|^2)^{a}e^{\hat{\phi}_k}}{\lambda_k|x|^2+1}dx\\
\leq~&\int_{|x|\leq R}\frac{(1+s^2_k|x|^2)^{\alpha}(\frac{1}{\lambda_k s_k^2}+|x|^2)^a e^{\eta_k(x)}}{\lambda_k s^2_k|x|^2+1}dx
+\frac{\rho}{\lambda_k s^2_kR^2+1}.
%\leq~&C\int_{|x|\leq R}\frac{1}{\lambda_k s^2_k|x|^2+1}+\frac{\rho}{\lambda_k s^2_k R^2+1}.
\end{aligned}
\end{equation*}
So, we can use (\ref{new35}) and  the dominated converge theorem to see that the right hand side of the above equality tends to $0$, as $k\to+\infty$. As a consequence, from (\ref{new26}) we find that $\beta_{\infty} =8\pi(1+a)$, and so $\rho=8\pi(2+a)=8\pi\alpha,$ (by recalling (\ref{new4})) in contradiction to (\ref{0.33}).

Therefore Claim 3 is establish, and in view of (\ref{new36}) and well known elliptic estimates, along a subsequence, we have:
\begin{equation*}
\eta_k \rightarrow \eta\ \ \text{uniformly in}\ C^2_{\mathrm{loc}}(\mathbb{R}^2),
\end{equation*}
with $\eta$ satisfying:
\begin{equation}
\label{new37}
\begin{cases}
\Delta\eta+(\frac{1}{L}+|x|^2)^{a} e^{\eta} = 0\ \ \mathrm{in}\ \mathbb{R}^2,\\
\rho_{\infty}=\int_{\mathbb{R}^2}(\frac{1}{L}+|x|^2)^{a}e^{\eta} <+\infty.
\end{cases}
\end{equation}
By recalling (\ref{new21}), we can use for problem (\ref{new37}) (after suitable scaling) the part (ii) of Proposition \ref{pr2.new} and conclude that $\eta$ is radially symmetric about the origin, where it attains its maximum value. As a consequence, $\frac{\hat{x}_k}{s_k} \rightarrow 0$, and so $|\hat{x}_k|\sqrt{\lambda_k} \rightarrow 0$, as $k \rightarrow +\infty$. Furthermore, (\ref{new37}) is solvable if and only if $\rho_{\infty}\in(8\pi(1+a),8\pi)$, and by means of a Pohozaev identity, (as in \cite{tt}) we also know that,
\begin{equation}
\label{new38}
\rho_{\infty}(\frac{\rho_{\infty}}{4\pi}-2(a+1))=-2a\int_{\mathbb{R}^2}\frac{(\frac{1}{L}+|x|^2)^a e^{\eta}}{(L|x|^2 +1)}dx.
\end{equation}
With this information, we can argue as above to estimate the integral term in (\ref{new26}). Indeed for fixed $r>0$ sufficiently small  and by taking a large $R>0$ we obtain:
\begin{equation*}
\lim_{k \rightarrow +\infty}\int_{|x|< r}\frac{(1+|x|^2)^{\alpha}(\frac{1}{\lambda_k}+|x|^2)^{a}e^{\hat\phi_k(x)}}{\lambda_k |x|^2 +1}dx
=\int_{|x|\leq R}\frac{(\frac{1}{L}+|x|^2)^{a}}{L|x|^2 + 1}e^{\eta} + O(\frac{1}{R^2}).
\end{equation*}
Thus, by letting $R \rightarrow +\infty$, from (\ref{new26}) and (\ref{new38}), we conclude:
\begin{equation*}
\beta_{\infty} (\frac{\beta_\infty}{4\pi}-2(a+1))=-2a\int_{\mathbb{R}^2}\frac{(\frac{1}{L}+|x|^2)^a}{L|x|^2 + 1}e^{\eta(x)}dx
=\rho_{\infty}(\frac{\rho_{\infty}}{4\pi} - 2(a+1)).
\end{equation*}
Hence, by using (\ref{new27}), we derive that,
\begin{equation}\label{new39}
\beta_{\infty} = \rho_{\infty}.
\end{equation}
At this point, by Claim 3 and (\ref{new39}), for fixed $r>0$ and $\varepsilon>0$ small we may conclude:
\begin{equation*}
\begin{aligned}
\int_{\frac{1}{r}\leq|x|\leq\varepsilon (\sqrt{\lambda_k})}(1+|x|^2)^{\alpha}(1+\frac{|x|^2}{\lambda_k})^{a}e^{\phi_k}
=\int_{\frac{1}{\varepsilon (\sqrt{\lambda_k})}\leq|x|\leq r}(1+|x|^2)^{\alpha}(\frac{1}{\lambda_k}+|x|^2)^{a}e^{\hat\phi_k}\to0,
\end{aligned}
\end{equation*}
as $k\to+\infty$ and $\varepsilon\to0^+.$ In turn, for fixed $r > 0$ and $\varepsilon > 0$ sufficiently small, we have:
\begin{equation*}
\begin{aligned}
\int_{|x|\leq\varepsilon}(\frac{1}{\lambda_k}+|x|^2)^{\alpha}(1+|x|^2)^{a}e^{\xi_k}=~&
\int_{|x|<\frac{1}{r}}(1+|x|^2)^{\alpha}(1+\frac{|x|^2}{\lambda_k})^{a}e^{\phi_k}dx\\
&+\int_{\frac{1}{r}\leq|x|\leq\varepsilon(\sqrt{\lambda_k})}(1+|x|^2)^{\alpha}(1+\frac{|x|^2}{\lambda_k})^{a}e^{\phi_k}dx \to 8\pi,
\end{aligned}
\end{equation*}
as $k \rightarrow +\infty$ and $\varepsilon \rightarrow 0^+$. As a consequence, we have $m_0 = 8\pi$ in (\ref{new214a}). Therefore, the mass distribution of $\xi_k$ cannot "concentrate", as otherwise $\rho=8\pi$. In other words, $\xi_k$ must be uniformly bounded from below, and we can use elliptic estimates locally in $\mathbb{R}^2 \backslash \{ 0\}$ to arrive at the conclusion of Theorem \ref{newth2.1}.
\end{proof}

On the basis of Theorem \ref{newth2.1}, for the original sequence $v_k$ satisfying (\ref{new3}) and (\ref{new5}), we obtain that,
\begin{equation*}
v_k \rightarrow v\ \ \text{uniformly in}\ C^2_{\mathrm{loc}}(\mathbb{R}^2\setminus\{0\})\ ~\mbox{as}~k \to+\infty
\end{equation*}
with $v$ satisfying:
\begin{equation*}
\begin{cases}
\Delta v + (1+|x|^2)^{a}e^{v} = 4\pi(\alpha-2)\delta_0\ \ \text{in}\ \mathbb{R}^2,\\
\int_{\mathbb{R}^2}(1+|x|^2)^{a}e^{v}=\rho-8\pi,
\end{cases}
\end{equation*}
and it suffices to pull back (via (\ref{new2})) such information to $u_k$ in order to derive the statement of Theorem B.2.

\section{The proof of Theorem 1.1}
\noindent{\em Proof of Theorem \ref{th1.1}.} To establish Theorem \ref{th1.1} (and as a consequence, Theorem B.1) we start by showing that $u^*_t$ must blow-up {\em only} at the point $\mathfrak{q}$ and "mass concentration" {\em can not} occur.

To this purpose, we denote by $\Sigma \neq \emptyset$ the blow-up set of $u_t^*$. Notice that around any possible point $p\in \Sigma\backslash\{\mathfrak{q }\}$ we can apply the blow-up analysis of \cite{bt1}, \cite{bm} and \cite{ls}. Therefore, if we suppose (by contradiction) that $\Sigma\backslash\{\mathfrak{q}\}\neq\emptyset$ then we have that "concentration" must occur and $\sigma_p = 4$ if $p \in \Sigma\backslash \{\mathfrak{q},q_{d+1},\cdots,q_N \}$ while $\sigma_p = 4(\alpha_i +1)$ if $p = q_i$ for some $i \in\{d+1,\cdots,N\}$. By assumption: $\alpha_i\in\mathbb{N}$, $\forall i\in\{1,\cdots,N\}$ and $\rho \notin 8 \pi \mathbb{N},$ hence we see that necessarily $\mathfrak{q}\in \Sigma.$ Indeed, in case $\mathfrak{q}\notin\Sigma$ then we would find that
$\rho=8\pi n,\ \ \text{for some}\ n \in \mathbb{N},$ which is impossible. Thus,
\begin{equation}
\label{3.0.a}
\rho=\sigma_{\mathfrak{q}} + 8\pi n,\ \ \text{for some}\ n \in \mathbb{N}.
\end{equation}

Furthermore, around $\mathfrak{q}$ we can use isothermal coordinates and from the sequence: $u_t-\text{log}\int_{M}h e^{u_t-G_t^{(2)}}$, we obtain (as indicated in the Introduction) a local sequence $\bar u_t$ (introduced in (\ref{eq116})) which satisfies all the assumptions of Theorem 1.4 of  \cite{llwz}. As a consequence, we derive the following about the local blow up mass $\sigma_{\mathfrak{q}}$ at $\mathfrak{q}$ :
\begin{equation}
\label{3.0.b}
\sigma_{\mathfrak{q}}=8\pi m~\mbox{with}~m\in \mathbb{N}~\mbox{and}~1 \leq m\leq [\frac{1}{2}\sum_{i=1}^d \alpha_i].
\end{equation}
Again (\ref {3.0.a}) and (\ref {3.0.b})  yield  to a contradiction, since by assumption $\rho \notin 8\pi \mathbb{N}.$  Consequently, $S\setminus\{\mathfrak{q}\} = \emptyset$ so that $u^*_t$ can blow-up only at $\mathfrak{q}$ with local blow up mass
$\sigma_{\mathfrak{q}}$ as specified in (\ref{3.0.b}). Therefore mass concentration cannot occur, and  $u^*_t$ is bounded uniformly in $L_{\mathrm{loc}}^{\infty}(M\backslash \{ \mathfrak{q} \})$.

At this point the desired conclusion follows by well known elliptic estimates; which imply that (along a subsequence),
\begin{equation*}
u^*_t\rightarrow w^*\ \ \text{uniformly in}\ C_{\mathrm{loc}}^2(M\backslash \{\mathfrak{q} \}),\ \text{as}\ t \rightarrow 0.
\end{equation*}
Moreover, with abuse of notation, if we denote by $B_r(\mathfrak{q})$ the ball of center $\mathfrak{q}$ and (small) radius $r>0$ with respect to the Riemannian metric in $M$, then by (\ref{0.31}) we find:
\begin{equation*}
\begin{aligned}
\int_{M}he^{u^*_t} &=\int_{M \backslash B_r(\mathfrak{q})}h e^{u^*_t}+\int_{M}he^{u^*_t}(\int_{B_r(\mathfrak{q})}\frac{h e^{u^*_t}}{\int_{M}h e^{u^*_t}})\\
&=\int_{M\backslash B_r(\mathfrak{q})}he^{u^*_t}+\int_{M}he^{u^*_t}(\frac{8\pi m}{\rho}+\varepsilon_{r,t}),
\end{aligned}
\end{equation*}
with $\lim_{r\rightarrow 0^+}\lim_{t \rightarrow 0}\varepsilon_{r,t}=0.$ In other words
\begin{equation*}
(\frac{\rho-8\pi m}{\rho}-\varepsilon_{r,t})\int\limits_{M}he^{u^*_t} =\int_{M\backslash B_r(\mathfrak{q})}he^{u^*_t},
\end{equation*}
and by passing to the limit as $t\rightarrow 0$ and then as $r\rightarrow 0^+$ we find,
\begin{equation*}
\lim_{t\to 0}\int_{M} h e^{u^*_t}=\frac{\rho}{\rho-8\pi m}\int_{M} h e^{w^*},
\end{equation*}
as claimed. As a consequence,
\begin{equation*}
\rho \frac{he^{u^*_t}}{\int_M he^{u^*_t}} \rightharpoonup 8\pi m \delta_{\mathfrak{q}} + (\rho-8\pi m)\frac{he^{w^*}}{\int_M he^{w^*}},
\end{equation*}
weakly in the sense of measure. Thus, $w^*$ satisfies (in the sense of distribution):
\begin{equation*}
\begin{aligned}
\Delta w^*+(\rho-8m\pi)(\frac{he^{w^*}}{\int_M he^{w^*}}-1)=4\pi(\sum\limits_{i=1}^d \alpha_i-2m)(\delta_\mathfrak{q} -1)
+4\pi\sum\limits_{i=d+1}^N\alpha_i(\delta_{q_i}-1).
\end{aligned}
\end{equation*}
Now we are going to show $\int_Mw^*(x)\mathrm{d}x=0$. From the Green's representation formula, we have
\begin{equation}
\label{l-3.4}
u_t(x)=\int_M\rho G(x,z)\frac{h^*(z)e^{u_t^*(z)}}{\int_Mh^*e^{u_t^*}}\mathrm{d}z.
\end{equation}
By taking the limit $t\to0$ in (\ref{l-3.4}), we get for any $x\in M\setminus\{0\},$
\begin{equation}
\label{l-3.5}
w(x)+8m\pi G(x,0)=(\rho-8m\pi)\frac{\int_MG(x,z)h(z)e^{w_m(z)}dz}{\int_Mhe^{w_m}}+8m\pi G(x,0).
\end{equation}
By integrating (\ref{l-3.5}) on $M$, we see that
\begin{align*}
\int_Mw(x)dx=~&(\rho-8m\pi)\frac{\int_M\left(\int_MG(x,z)h(z)e^{w_m(z)}\mathrm{d}z\right)\mathrm{d}x}{\int_Mhe^{w_m}}\\
=~&(\rho-8m\pi)\frac{\int_M\left(\int_MG(x,z)\mathrm{d}x\right)h(z)e^{w_m(z)}\mathrm{d}z}{\int_Mhe^{w_m}}.
\end{align*}
Since $\int_MG(x,z)\mathrm{d}z=0$, we obtain $\int_Mw^*(x)\mathrm{d}x=0,$ and the proof is completed. $\hfill\Box$
\bigskip

\section{The bubbling analysis of Equation (\ref{eq115}) and the proof of \\Theorem \ref{th1.2}}
In this section we study the bubbling behavior of the solution $\bar u_t$ satisfying:
\begin{equation}
\label{eq400}
\Delta \bar u_t + h_1(x)|x-t\underline{e}|^{2\alpha_1}|x+t\underline{e}|^{2\alpha_2} e^{\bar u_t} = 0,\ \ \text{in}\ B_1(0).
\end{equation}
\begin{equation}
\label{eq401}
\begin{aligned}
|\bar u_t(x)-\bar u_t(y)|\leq C,\ \forall\ x,y \in \partial B_1(x),\quad \int_{B_1(0)}h_1(x)|x-t\underline{e}|^{2\alpha_1}|x+t\underline{e}|^{2\alpha_2}e^{\bar u_t} \leq C,
\end{aligned}
\end{equation}
and
\begin{equation}
\label{eq402}
0<h_1(x)\in C^2(\overline{B}_1(0)),\ \alpha_1,\alpha_2\in\mathbb{N}:\ |\alpha_1-\alpha_2|\leq 1.
\end{equation}
Furthermore, we let
\begin{equation}
\label{eq403}
v_t(y)=\bar u_t(ty)+2(1+\alpha_1+\alpha_2)\log\ t,
\end{equation}
which satisfies:
\begin{equation}
\label{eq404}
\Delta v_t+h_1(ty)|y-\underline{e}|^{2\alpha_1}|y+\underline{e}|^{2\alpha_2}e^{v_t}=0.
\end{equation}
We assume that,
\begin{equation}
\label{eq405}
\mbox{the origin is the only blow-up point of}\ \bar u_t\ \mbox{in}\ B_1(0),
\end{equation}
and set
\begin{equation}
\label{eq406}
\sigma_u=\frac{1}{2\pi}\lim_{r\to0^+}\lim_{t\to0}\int_{B_r(0)}h_1(x)|x-t\underline{e}|^{2\alpha_1}|x+t\underline{e}|^{2\alpha_2}e^{\bar u_t},
\end{equation}
and
\begin{align}
\label{2.6}
m_v=\frac{1}{2\pi}\lim_{R\to+\infty}\lim_{t\rightarrow0}\int_{B_R(0)}h_1(ty)|y-\underline{e}|^{2\alpha_1}|y+\underline{e}|^{2\alpha_2}e^{v_t}dy.
\end{align}
For the quantity $\sigma_u$ and $m_v$, we have the following relation:

\begin{proposition}
\label{pr4.1}
Let $\sigma_u$ and $m_v$ be defined in (\ref{eq406}) and (\ref{2.6}), then we have
\begin{equation}
\label{2.7}
\sigma_u^2-m_v^2=4(1+\alpha_1+\alpha_2)(\sigma_u-m_v).
\end{equation}
\end{proposition}

\noindent{\bf Remark:} Proposition \ref{pr4.1} holds for general positive constants $\alpha_1,\alpha_2.$
\medskip

In order to prove Proposition \ref{pr4.1}, we shall use the following estimates established in \cite{llwz}
\medskip

\noindent {\bf Lemma 4.A.} \cite[Lemma 4.2]{llwz} {\em Let $\bar u_t$ satisfy (\ref{eq400}) and $v_t$ be defined in (\ref{eq403}). Then,
\begin{enumerate}
  \item For $r>0$ sufficiently small,
  \begin{equation}
  \label{2.8}
  \nabla \bar u_t=-\sigma_u\frac{x}{|x|^2}+\nabla\phi ~\mathrm{in}~C_{\mathrm{loc}}^0(B_r(0)\setminus\{0\}) \ \ \text{as}\ t\rightarrow 0,
  \end{equation}
  for some $\phi\in C^1(B_r(0)).$
  \item
  \begin{equation}
  \label{2.9}
  \nabla v_t=-(m_v+o_R(1))\frac{x}{|x|^2}~\mathrm{for}~|x|\geq R\ \text{and}\ t\rightarrow 0,
  \end{equation}
  where $o_R(1)\to0$ as $R\to+\infty.$
\end{enumerate}}
\medskip

\noindent {\em Proof of Proposition \ref{pr4.1}.}
We multiply the equation (\ref{eq404}) by $y\cdot\nabla v_t$ and integrate both sides over $B_{\frac{r}{t}}(0)\setminus B_R(0)$, we get that
\begin{equation}
\begin{aligned}
\label{poho}
&r\int_{\partial B_r(0)}\left(|\partial_{\nu}\bar u_t|^2-\frac12|\nabla\bar u_t|^2
+h_1(x)|x-t\underline{e}|^{2\alpha_1}|x+t\underline{e}|^{2\alpha_2}e^{\bar u_t}\right)ds_x
\nonumber\\
&-R\int_{\partial B_R(0)}\left(|\partial_{\nu}v_t|^2-\frac12|\nabla v_t|^2+ h_1(ty)|y-\underline{e}|^{2\alpha_1}|y+\underline{e}|^{2\alpha_2}e^{v_t}\right)ds_y
\nonumber\\
=&\int_{B_{\frac{r}{t}}(0)\setminus B_R(0)}2\left(1+\alpha_1\frac{y\cdot(y-\underline{e})}{|y-\underline{e}|^2}
+\alpha_2\frac{y\cdot(y+\underline{e})}{|y+\underline{e}|^2}\right)
h_1(ty)|y-\underline{e}|^{2\alpha_1}|y+\underline{e}|^{2\alpha_2}e^{v_t}dy\nonumber\\
&+\int_{B_{\frac{r}{t}}(0)\setminus B_R(0)}\left(ty\cdot\left(\frac{\nabla_xh_1(ty)}{h_1(ty)}\right)\right)
h_1(ty)|y-\underline{e}|^{2\alpha_1}|y+\underline{e}|^{2\alpha_2}e^{v_t}dy,
\end{aligned}
\end{equation}
where we used $x\cdot\nabla\bar u_t=y\cdot\nabla v_t$ and $\partial_{\nu}$ denotes the derivative along the normal direction. When $R\to\infty,$ we note that
\begin{equation}
\label{alpha}
\alpha_1\frac{y\cdot(y-\underline{e})}{|y-\underline{e}|^2}
+\alpha_2\frac{y\cdot(y+\underline{e})}{|y+\underline{e}|^2}\rightarrow\alpha_1+\alpha_2,~\mathrm{for}~|y|\geq R.
\end{equation}
Using (\ref{2.8}), (\ref{2.9}), (\ref{poho}) and (\ref{alpha}), we can get (\ref{2.7}). \hfill $\square$
\medskip

Before analyzing the behaviour of $v_t$, we recall the following theorem from \cite[Theorem 2.1]{lwz}.
\medskip

\noindent {\bf Theorem 4.A.}\quad {\em
Let $u$ be a solution of the following equation
\begin{equation}
\label{2.a.1}
\Delta u+e^u=\sum_{i=1}^l4\pi\alpha_i\delta_{p_i}~\mathrm{in}~\mathbb{R}^2,\quad \int_{\mathbb{R}^2}e^udx<+\infty.
\end{equation}
Then
\begin{equation}
\label{2.a.2}
\int_{\mathbb{R}^2}e^u\mathrm{d}x=4\pi\left(\sum_{i=1}^l\alpha_i+\frac{\alpha}{2}\right)\in8\pi\mathbb{N},~\mathrm{with}~\alpha>2.
\end{equation}}
\medskip

Going back to the study of $v_t$, we restate Theorem \ref{th1.2} as follows:
\begin{proposition}
\label{pr4.2}
Let $\bar u_t$ be a sequence of blow up solutions of (\ref{eq400}), and assume that (\ref{eq401}), (\ref{eq402}) and (\ref{eq405}) hold. If $v_t$, $\sigma_u$ and $m_v$ are defined in (\ref{eq403}), (\ref{eq406}) and (\ref{2.6}) respectively, then $v_t$ blows up at $m$-disctinct points $\{p_1,\cdots,p_m\}\subset\mathbb{R}^2\setminus\{\underline{e},-\underline{e}\}$ and $\sigma_u=m_v=4m$ with $1\leq m\leq\min\{\alpha_1,\alpha_2\}.$
 \end{proposition}

\begin{proof}
At first, we point out a fact which will be frequently used in the sequel. It was used first in \cite{bt1} and more recently in \cite{ttt} to carry out the blow-up analysis of cosmic strings. It gives us a criterion to recognise when "concentration" is bound to occur. We claim that, if $\sigma_u\geq2\alpha_1+2\alpha_2+2$, then $\bar u_t$ must concentrate, i.e.,
$$\bar u_t\rightarrow-\infty\ \text{uniformly in}\ C_{\mathrm{loc}}^0 (B_1 \backslash \{ 0 \}).$$
We prove it by contradiction and suppose that $\bar u_t$ is uniformly bounded in $L^{\infty}(\partial B_r(0)),$ for some $r \in (0,1)$. Let $z_t$ be the solution of
\begin{align}
\label{2.13}
\left\{\begin{array}{l}
\Delta z_t+h_1(x)|x-t\underline{e}|^{2\alpha_1}|x+t\underline{e}|^{2\alpha_2}e^{\bar u_t}=0~\mathrm{in}~B_r(0),\\
z_t=-|\bar u_t|~\mathrm{on}~\partial B_r(0).
\end{array}\right.
\end{align}
By the maximum principle, $u_t\geq z_t$ in $B_r(0)$, and in particular,
\begin{equation}
\label{2.14}
\begin{aligned}
&\int_{B_r(0)}h_1(x)|x-t\underline{e}|^{2\alpha_1}|x+t\underline{e}|^{2\alpha_2}e^{z_t}\\
&\leq\int_{B_r(0)} h_1(x)|x-t\underline{e}|^{2\alpha_1}|x+t\underline{e}|^{2\alpha_2}e^{\bar u_t}< C.
\end{aligned}
\end{equation}
On the other hand $z_t\rightarrow z$ uniformly on compact subsets of $B_r(0)\setminus\{0\}$, and $z$ satisfies
\begin{align}
\label{2.15}
-\Delta z=\mu~\mathrm{in}~B_r(0),\quad z \in L^{\infty}(\partial B_r(0)).
\end{align}
in the sense of distribution, with the measure $\mu$ such that, $\mu\geq 2\pi\sigma_u\delta_0$. Therefore
$$z(x)\geq\sigma_u\log\frac{1}{|x|}+O(1),~\mathrm{as}~x\rightarrow 0.$$
As a consequence, $e^z\geq\frac{C}{|x|^{\sigma_u}}$ and since (by the assumption) $\sigma_u\geq2(\alpha_1+\alpha_2+1)$ we get a contradiction to (\ref{2.14}).
Thus, the claim is proved.

Let us go back to the study of $v_t$ (in (\ref{eq404})). From the classical work \cite{bm} by Brezis and Merle, we have three possibilities for the behavior of $v_t:$
\begin{enumerate}
  \item $v_t$ is bounded in $L_{\mathrm{loc}}^{\infty}(\mathbb{R}^2),$
  \item $v_t\to-\infty$ uniformly on compact subsets of $\mathbb{R}^2$,
  \item there exists a finite set $S_v=\{p_1,p_2,\cdots,p_m\}\subset\mathbb{R}^2$ such that $v_t$ blows up at $p_j,~j=1,\cdots,m$, $v_t\rightarrow-\infty$ in $L^{\infty}_{\mathrm{loc}}(\mathbb{R}^2\setminus\{p_1,p_2,\cdots,p_m\})$ and
  $$h_1(ty)|y-\underline{e}|^{2\alpha_1}|y+\underline{e}|^{2\alpha_2}e^{v}\rightharpoonup 8\pi\sum_{j=1}^m\beta_j\delta_{p_j},$$
  where $\beta_j=1$ if $p_j\notin\{\underline{e},-\underline{e}\}$, $\beta_j=1+\alpha_1$ if $p_j=\underline{e}$ and $\beta_j=1+\alpha_2$ if $p_j=-\underline{e}.$
\end{enumerate}
We shall rule out first that (1) and (2) can occur.
\medskip

Indeed, concerning (1) we see that in case $v_t$ is locally bounded in $L_{\mathrm{loc}}^\infty(\mathbb{R}^2),$ then (along a sequence) we can get that $v_t$ converges to $v_0$ locally in $\mathbb{R}^2,$ with $v_0$  satisfying:
\begin{equation}
\label{2.16}
\Delta v_0+ h_1(0)|y-\underline{e}|^{2\alpha_1}|y+\underline{e}|^{2\alpha_2}e^{v_0}=0,
\end{equation}
where $h_1(0) > 0$ and
$$\int_{\mathbb{R}^2}|y-\underline{e}|^{2\alpha_1}|y+\underline{e}|^{2\alpha_2}e^{v_0}<+\infty.$$
From Theorem 4.A, we derive that
\begin{equation*}
\int_{\mathbb{R}^2}h(0)|y-\underline{e}|^{2\alpha_1}|y+\underline{e}|^{2\alpha_2}e^{v_0}=4\pi(\alpha_1+\alpha_2+\frac{\theta}{2})
\end{equation*}
for some positive integer $\theta>2$ and $\alpha_1+\alpha_2+\frac{\theta}{2}\in2\mathbb{N}$. As a consequence, we get that $m_v=2\alpha_1+2\alpha_2+\theta$. Using (\ref{2.7}) and the simple fact $\sigma_u\geq m_v$, we deduce that necessarily $$\sigma_u=m_v=2\alpha_1+2\alpha_2+\theta.$$ In addition, by virtue of the above claim we can also get that $\bar u_t$ concentrates. So the mass of $u^*_t$ must concentrate and $\rho=2\pi\sigma_u= 2\pi(2\alpha_1+2\alpha_2+\theta)\in 8\pi\mathbb{N},$ which contradicts our assumption on $\rho$. Therefore, (1) is ruled out.
\medskip

If $v_t\to-\infty$, then $m_v=0$, and by (\ref{2.7}), we obtain $\sigma_u=4\alpha_1+4\alpha_2+4$. Obviously, $\sigma_u>2\alpha_1+2\alpha_2+2$, and $\bar u_t$ concentrates. Again we derive that $\rho\in 8\pi\mathbb{N}$ and this is a contradiction. Thus, also alternative (2) is ruled out.

In conclusion we see that $v_t$ blows up at finite points $S_v=\{p_1,\cdots,p_m\},$ and we shall check that $S_v\cap\{\underline{e},-\underline{e}\}=\emptyset.$ Indeed, if $\underline{e}\in S_v$ or $-\underline{e}\in S_v$, then by virtue of (\ref{eq402}) we find that,
$$m_v\geq 4+4\min\{\alpha_1,\alpha_2\}\geq2+2\alpha_1+2\alpha_2,$$
As a consequence, by using (\ref{2.7}), we get $$\sigma_u=m_v\geq2+2\alpha_1+2\alpha_2,$$ which implies that $\bar u_t$ concentrates and $\rho\in 8\pi\mathbb{N}$, a contradiction. Hence
$$S_v\cap\{\underline{e},-\underline{e}\}=\emptyset.$$
Next we check that $\sigma_u = m_v$ and $1 \leq m \leq \min\{ \alpha_1,\alpha_2 \}$. Indeed, we know that $m_v=4m$. If (by contradiction) we assume that $\sigma_u > m_v$ then by (\ref{2.7}), we derive that,
$$4m <\sigma_u=4(\alpha_1+\alpha_2+1)-m_v=4(\alpha_1+\alpha_2-m+1),$$
As a consequence, $2m<\alpha_1+\alpha_2+1$ which implies that $\sigma_u>2(\alpha_1 +\alpha_2 +1)$. So concentration must occur and again we find that $\rho \in 8\pi\mathbb{N}$, which is impossible. In conclusion $\sigma_u = m_v = 4m$ and concentration {\em cannot} occur. So necessarily
$4m =\sigma_u<2(\alpha_1 + \alpha_2 +1$),\ which in view of (\ref{eq402}) gives,\ $1 \leq m \leq \min \{ \alpha_1,\alpha_2 \}$\ as claimed.
\end{proof}

\section{The proof of Theorem \ref{th1.3} and Theorem \ref{th1.4}}
From the discussion of the previous two sections, we have obtained a description of $\bar u_t$ away from the origin and showed that the scaled function $v_t$ must blow up. In this section, we shall study the behavior of $v_t$ around each blow up point. More precisely, we shall derive a relation between $t$ and the maximal value of $v_t.$ From Proposition \ref{pr4.2} we know that $v_t$ blows up at $m$ points $\{p_1,p_2,\cdots,p_m\}$ different from $\pm \underline{e},$ and $\sigma_u=m_v=4m$, with
$$m\in\{1, \cdots,\min\{\alpha_1,\alpha_2\}\}.$$
We recall that $r_0$ and $R$ are two positive fixed constants such that $B_{4r_0}(p_i)\cap B_{4r_0}(p_j)=\emptyset,i\neq j$ and $\bigcup_{i=1}^mB_{4r_0}(p_{i})\subset B_{R}(0)$, and we have set
\begin{equation}
\label{3.1}
\lambda_{t,i}=\max_{B_{r_0}(p_i)}v_t(y)=v_t(p_{t,i})\quad\mathrm{and}\quad\lambda_t=\max_{i}\lambda_{t,i}.
\end{equation}
In the following proposition we shall derive a first rough estimation between $\lambda_t$ and $t$.

\begin{proposition}
\label{pr3.1}
Let $\lambda_t$ be defined above, then the following holds:
\begin{align}
\label{3.2}
\lambda_t+(2+2\alpha_1+2\alpha_2-4m)\log t=O(1),\ \text{as}\ t \rightarrow 0.
\end{align}
\end{proposition}

\begin{proof} We observe that $v_t$ admits bounded oscillation in $B_{2R}(0)\setminus\bigcup_{i=1}^m B_{r_0}(p_i),$ in the sense that,
\begin{equation}
\label{3.3}
|v_t(y_1)-v_t(y_2)|\leq C,~\forall y_1,y_2\in B_{R}(0)\setminus\bigcup_{i=1}^m B_{r_0}(p_i),
\end{equation}
which we may check as in \cite{bclt}. As a consequence we can use the pointwise estimate established in \cite[Theorem 1.1]{l}, to obtain that,
\begin{equation}
\label{3.4}
\left|v_t(y)-\log\frac{e^{\lambda_{t,i}}}{(1+\frac{h_1(tp_{t,i})}{8}e^{\lambda_{t,i}}
|p_{t,i}-\underline{e}|^{2\alpha_1}|p_{t,i}+\underline{e}|^{2\alpha_2}|y-p_{t,i}|^2)^2}\right|\leq C
\end{equation}
for $|y-p_{t,i}|\leq r_0$. From (\ref{3.4}), we deduce that $\lambda_t=\lambda_{t,i}+O(1)$ and $v_t=-\lambda_t+O(1)$ for $y\in \partial B_R(0).$ Therefore, for $x\in\partial B_{Rt}(0)$ we have
\begin{align}
\label{3.5}
\bar u_t=-\lambda_t-(2+2\alpha_1+2\alpha_2)\log t+O(1).
\end{align}
On the other hand, by using the Green's representation formula for $\bar u_t$ in $B_1(0)$, for $x\in\partial B_{Rt}(0)$ we have:
\begin{align}
\label{3.6}
\bar u_t(x) = -\frac{1}{2\pi} \int_{B_1(0)}\text{log}\ |x-z| h_{1,t}(z) e^{\bar u_t(z)} + O(1).
\end{align}
For $t \rightarrow 0$, we decompose:
\begin{align*}
B_1(0)=B_{2Rt}(0)\cup(B_{r}(0)\setminus B_{2Rt}(0))\cup( B_1(0)\setminus B_{r}(0))=I_1\cup I_2\cup I_3,
\end{align*}
with $r>0$ sufficiently small, so that
\begin{align}
\label{3.7}
\int_{B_{r}(0)}h_{1,t}(x) e^{\bar u_t(x)}dx=8m\pi+o_{r}(1),~\mbox{where}~o_r(1)\to0~\mbox{as}~r\to0.
\end{align}
For $z\in I_3$, we have $\bar u_t(z)\leq C_{r}$, and so,
\begin{equation}
\label{3.8}
\left|\int_{I_3}-\frac{1}{2\pi}\log|x-z| h_{1,t}e^{\bar u_t(z)}\right|\leq C_{r}.
\end{equation}
When $z\in I_2$, and $x\in \partial B_{Rt}(0)$, we have $|x-z|\geq O(|z|)$, and so,
\begin{equation}
\label{3.9}
\left|\int_{I_2}\frac{1}{2\pi}\log|x-z|h_{1,t}e^{\bar u_t(z)} \right| \leq o_{r}(1)(\log \frac{1}{t}).
\end{equation}
For the integration over $I_1,$ we can use (\ref{3.4}) and after scaling, obtain (as in the step 2 in the proof of \cite[Theorem 1.1]{bclt}) that,
\begin{equation}
\label{3.10}
\int_{B_{r_0t}(tp_{t,i})} h_{1,t}e^{\bar u_t(z)}=8\pi\left(1+O\left(\lambda_t^{-1}\right)\right),
\end{equation}
which together with (\ref{3.3}) gives that
\begin{equation}
\label{3.m}
\int_{I_1}h_{1,t}e^{\bar u_t(z)}=8m\pi\left(1+O\left(\lambda_t^{-1}\right)\right).
\end{equation}
Therefore by letting $x=Rte'\in \partial B_{Rt}(0)$ with $e'\in\mathbb{S}^1$, we  get,
\begin{equation}
\label{3.11}
\begin{aligned}
&~\int_{B_{2Rt}(0)}\frac{1}{2\pi}\log|x-z|h_{1,t}e^{\bar u_t(z)}\mathrm{d}z\\
=&~\frac{1}{2\pi}\log t\int_{B_{2R}(0)}h_1(ty)|y-\underline{e}|^{2\alpha_1}|y+\underline{e}|^{2\alpha_2}e^{v_t}\mathrm{d}y\\
&~+\frac{1}{2\pi}\int_{B_{2R}(0)}\log|Re'-z|h_1(ty)|y-\underline{e}|^{2\alpha_1}|y+\underline{e}|^{2\alpha_2}e^{v_t}\mathrm{d}y\\
=&~4m\log t\left(1+O\left(\lambda_t^{-1}\right)\right)+O(1).
\end{aligned}
\end{equation}
Thus by using (\ref{3.5})-(\ref{3.11}), we have
\begin{align}
\label{3.12}
-\lambda_t-(2+2\alpha_1+2\alpha_2)\log t \leq -4m\log t\left(1+O\left(\lambda_t^{-1}\right)\right)+o_{r}(1)|\log t|+C_{r}.
\end{align}
Next we shall improve the estimate (\ref{3.12}). To this purpose, we prove first that,
\begin{align*}
\bar u_t+(4m+1)\log|x|\leq C~\mathrm{in}~B_{r}(0)\setminus B_{Rt}(0).
\end{align*}
Let
$$u_{t,1}(x)=\bar u_t(x)+(4m+1)\log|x|,$$
then we must show that: $u_{t,1}\leq C$ in $B_{r}(0)\setminus B_{Rt}(0)$. From the definition, we see that $u_{t,1}$ satisfies:
\begin{align}
\label{3.13}
\Delta u_{t,1}+ h_{1,t}(x)e^{\bar u_t}(x)=0,\ \text{for}\ x \in B_r(0) \backslash B_{Rt}(0).
\end{align}
We can choose $r>0$ sufficiently small, so that:
$$\int_{B_{r}(0)\setminus B_{Rt}(0)}\left| h_{1,t}e^{\bar u_t} \right|\leq \frac12\pi.$$
Since for $x\in\partial\left(B_{r}(0)\setminus B_{2Rt}(0)\right)$ we can use (\ref{3.12}) with $m \in \mathbb{N}$ and $1 \leq m\leq\min\{\alpha_1,\alpha_2\}$ to get
\begin{align}
\label{3.14}
u_{t,1}(x)=\bar u_t(x)+(4m+1)\log|x|\leq C.
\end{align}
At this point, we can argue as in \cite[Corollary 4]{bm} to obtain that $u_{t,1}\leq C$ in $I_2.$ By applying this inequality together with the following information:
$$(\log|x-z|)h_{1,t}(z) = O((\log|z|)|z|^{2(\alpha_1 +\alpha_2)}),\ \forall\ z \in I_2,~\forall x\in\partial B_{Rt}(0),$$
we can further get that,
\begin{align}
\label{3.15}
\int_{I_2}\frac{1}{2\pi}\log|x-z|h_{1,t}e^{u_t}\mathrm{d}z=O(1).
\end{align}
As a consequence of (\ref{3.11}) and (\ref{3.15}), we have
\begin{align}
\label{3.16}
-\lambda_t-(2+2\alpha_1+2\alpha_2)\log t=-4m\log t(1+O(\lambda_t^{-1}))+O(1),
\end{align}
which gives,
\begin{align}
\label{3.17}
\lambda_t+(2+2\alpha_1+2\alpha_2-4m)\log t=O(1),
\end{align}
as claimed.
\end{proof}

\noindent {\em Proof of Theorem \ref{th1.3}.} The first conclusion is already proved in Proposition \ref{pr3.1}. Using (\ref{3.2}) and repeating the arguments which are used for proving $u_{t,1}$ is bounded from above, we can further show that
\begin{align*}
\bar u_t+4m\log|x|\leq C~\mathrm{in}~B_{r}(0)\setminus\bigcup_{i=1}^m B_{r_0t}(tp_{t,i}).
\end{align*}
In addition, by combining the above inequality and Theorem \ref{th1.1}, we can obtain
\begin{align}
\label{3.18}
u_t-G_t^{(2)}\leq C~\mathrm{in}~M\setminus B_{\lambda t}(\mathfrak{q})~\mbox{for}~\lambda~\mbox{sufficiently large}.
\end{align}
At this point the conclusion (ii) of Theorem \ref{th1.3} obviously holds. \hfill $\square$
\medskip

To proceed further, we recall that $y_{i,t}$ denotes the point in $M$ which corresponds to $tp_{t,i} \in B_1(0)$ and $U_{i,t}^{r_0}$ is the neighborhood of $y_{i,t}$ in $M$ corresponding to the ball $B_{tr_0}(t p_{t,i})$ under the isothermal coordinates around $\mathfrak{q}$ (centered at the origin). So $y_{i,t} \rightarrow \mathfrak{q}$ as $t \rightarrow 0$. Furthermore, we have set,
\begin{equation}
\label{eq520}
\rho_{t,i} = \int_{B_{r_0 t}(t p_{t,i})}h_{1,t}(x) e^{\bar u_t}\mathrm{d}x = \frac{\int_{U_{i,t}^{r_0}}\rho h e^{u_t-G_t^{(2)}}\mathrm{d}\sigma}{\int_{M}h e^{u_t-G_t^{(2)}}\mathrm{d}\sigma}.
\end{equation}
From Theorem \ref{th1.1} we have,
\begin{equation}
\label{4-eq-con}
u_t\to w+8m\pi G(x,0),
\end{equation}
with $w \in \mathring{H}^1(M)$ satisfying:
\begin{align}
\label{s.2}
\Delta w+(\rho-8m\pi)\left(\frac{he^{w+(8m-4\alpha_1-4\alpha_2)\pi G(x,0)}}{\int_Mhe^{w+(8m-4\alpha_1-4\alpha_2)\pi G(x,0)}}-1\right)=0.
\end{align}
To get a sharp estimate on $\rho_{t,i}-8\pi$, it is necessary to analyze the term:
\begin{align}
\label{s.3}
\tilde{\phi}_t=u_t-w-\sum_{i=1}^m\rho_{t,i}G(x,y_{i,t}).
\end{align}

\begin{proposition}
\label{pr3.2}
Suppose $w$ is a non-degenerate solution of (\ref{s.2}). Let $\rho_{t,i}$ and $\tilde{\phi}_t$ be defined in (\ref{eq520}) and (\ref{s.3}). Then
\begin{enumerate}
  \item [(i).] $\rho_{t,i}-8\pi=O(t^2\log t),$
  \item [(ii).] $\|\tilde\phi_t\|_{C^1(M\setminus\bigcup_{i=1}^m U_{i,t}^{2r_0})}= O(t\log t).$
\end{enumerate}
\end{proposition}

\begin{proof}
We shall prove the given estimates in the following steps.
\medskip

\noindent Step 1. We define $\|\tilde{\phi}_t\|_*=\|\tilde{\phi}_t\|_{C^1(M\setminus\bigcup_{i=1}^m U_{i,t}^{2r_0})}$ and show it is small.
\medskip

In the proof of Theorem \ref{th1.1}, we have shown that
\begin{align}
\label{s.6}
\int_Mhe^{u_t-G_t^{(2)}}=\frac{\rho}{\rho-8m\pi}\int_Mhe^{w_m}+o(1),
\end{align}
with $o(1)\to0$ as $t\to0.$ By Green's representation formula, for $x\in M\setminus B_{2r_0t}(tp_{t,i})$ we have,
\begin{equation}
\label{s.7}
\begin{aligned}
u_t(x)=~&\int_M\rho G(x,z)\left(\frac{he^{u_t-G_t^{(2)}}}{\int_Mhe^{u_t-G_t^{(2)}}}-1\right)\\
=~&\left(\int_{M\setminus B_{r}(0)}+\int_{B_r(0)\setminus\bigcup_{i=1}^m U_{i,t}^{r_0}}+\sum_{i=1}^m\int_{U_{i,t}^{r_0}}\right)
G(x,z)\frac{\rho he^{u_t-G_t^{(2)}}}{\int_Mhe^{u_t-G_t^{(2)}}}\\
=~&\int_{M\setminus B_r(0)}G(x,z)\frac{(\rho-8m\pi)he^{w_m}}{\int_Mhe^{w_m}}+\sum_{i=1}^m\rho_{t,i}G(x,y_{i,t})\\
&+\int_{B_{r}(0)\setminus\bigcup_{i=1}^m U_{i,t}^{2r_0}}
G(x,z)\frac{he^{u_t-G_t^{(2)}}}{\int_Mhe^{u_t-G_t^{(2)}}}+o(1)\\
=~&w(x)+\sum_{i=1}^m\rho_{t,i}G(x,y_{i,t})+o(1),
\end{aligned}
\end{equation}
where we have used (\ref{3.4}) and Proposition \ref{pr3.1}. Similarly, we can get
$$\nabla u_t(x)=\nabla w(x)+\sum_{i=1}^m\rho_{t,i}\nabla G(x,y_{i,t})+O(1).$$
Then $\|\t\phi_t\|_*=O(1).$ Next, we shall improve the estimate for $\nabla\t\phi_t.$ To this purpose, we have to get a better description on $v_t.$ Recall that $v_t$ satisfies
\begin{equation}
\label{s.9}
\Delta v_t+h_{1,t}(ty)e^{v_t(y)}=0,
\end{equation}
where
\begin{equation*}
h_{1,t}(ty)=\rho h(ty)|y-\underline{e}|^{2\alpha_1}|y+\underline{e}|^{2\alpha_2}e^{-R_t(y)+\psi_t(y)},
\end{equation*}
and $\psi_t(y)=\psi(ty)$ with $\psi(x)$ given in (\ref{def-psi}). Notice that $\|\psi_{t}\|_{C^1(B_{R}(0))}\leq Ct^2.$

On $\partial B_{2r_0}(p_{t,i})$, we have:
\begin{equation}
\label{s.10}
\begin{aligned}
v_{t}(y)=~&-\sum_{j=1}^m\frac{\rho_{t,j}}{2\pi}\log|y-p_{t,j}|
+\left(2+2\alpha_1+2\alpha_2-\sum_{j=1}^m\frac{\rho_{t,j}}{2\pi}\right)\log t+w(ty)\\
&+\sum_{j=1}^m\rho_{t,j}R(ty,tp_{t,j})-\log\int_Mhe^{u_t-G_t^{(2)}}+\tilde\phi_t(ty)-\psi_{t}(y)\\
=~&-\frac{\rho_{t,i}}{2\pi}\log|y-p_{t,i}|+\left(2+2\alpha_1+2\alpha_2-\frac{\rho_{t,i}}{2\pi}\right)\log t\\
&+G_{t,i}(ty)-\log\int_Mhe^{u_t-G_t^{(2)}}+\tilde\phi_t(ty)-\psi_{t}(y),
\end{aligned}
\end{equation}
where
$$G_{t,i}(ty)=\rho_{t,i}R(ty,y_{i,t})+\sum_{j\neq i}\rho_{t,j}G(ty,y_{j,t})+w(ty).$$
For any unit vector $\xi\in\mathbb{R}^2,$ we apply the Pohozaev identity to (\ref{s.9}) and obtain
\begin{equation}
\label{s.16}
\begin{aligned}
\int_{B_{2r_0}(p_{t,i})}(\xi\cdot\nabla h_{1,t})e^{v_{t}(y)}
=&\int_{\p B_{2r_0}(p_{t,i})}\left[(\nu\cdot\nabla v_{t})(\xi\cdot\nabla v_{t})-\frac12(\nu\cdot\xi)|\nabla v_{t}|^2\right]\\
&+\int_{\p B_{2r_0}(p_{t,i})}(\nu\cdot\xi)h_{1,t}e^{v_{t}},
\end{aligned}
\end{equation}
where $\nu$ denotes the unit normal of $\partial B_{2r_0}(p_{t,i})$. For the right hand side of (\ref{s.16}), we can use (\ref{s.10}), to find:
\begin{equation}
\label{s.17}
\begin{aligned}
&\int_{\p B_{2r_0}(p_{t,i})}\Big[(\nu\cdot\nabla v_{t})(\xi\cdot\nabla v_{t})-\frac12(\nu\cdot\xi)|\nabla v_{t}|^2\Big]\\
=&-t\rho_{t,i}\xi\cdot\nabla_xG_{t,i}(tp_{t,i})+O(t\|\tilde\phi_t\|_{*}+t^2).
\end{aligned}
\end{equation}
While, for the left hand side of (\ref{s.16}), we get
\begin{equation}
\label{s.18}
\begin{aligned}
&\int_{B_{2r_0}(p_{t,i})}\left(\xi\cdot\nabla h_{1,t}(ty)\right)e^{v_{t}(y)}\mathrm{d}y\\
&=\int_{B_{r_0}(p_{t,i})}\left(\xi\cdot\frac{\nabla h_{1,t}(ty)}{h_{1}(ty)}\right)h_{1}(ty)e^{v_{t}(y)}+O(t^2)\\
&=\int_{B_{r_0}(p_{t,i})}\xi\cdot(\frac{\nabla h_{1,t}(ty)}{h_{1,t}(ty)}-\frac{\nabla h_{1,t}(tp_{t,i})}{h_{1,t}(tp_{t,i})})h_{1,t}(ty)e^{v_{t}(y)}\\
&\quad+\int_{B_{r_0}(p_{t,i})}\left(\xi\cdot\frac{\nabla h_{1,t}(tp_{t,i})}{h_{1,t}(tp_{t,i})}\right) h_{1,t}(ty)e^{v_{t}(y)}+O(t^2)\\
&=\rho_{t,i}\left(\xi\cdot\frac{\nabla h_{1,t}(tp_{t,i})}{h_{1,t}(tp_{t,i})}\right)+O(t),
\end{aligned}
\end{equation}
where we used the Taylor expansion of $\log\frac{h_{1,t}(tx)}{h_{1,t}(tp_{t,i})}$. Thus, for any unit vector $\xi$, we have
\begin{align}
\label{s.19}
\xi\cdot\left(\nabla_y\log h_{1,t}(tp_{t,i})+t\nabla_xG_{t,i}(tp_{t,i})\right)=O(t+t\|\t\phi\|_*).
\end{align}
Using the expression of $h_{1,t}$, from (\ref{s.19}) we can obtain
\begin{equation}
\label{s.20n}
\begin{aligned}
&t\frac{\nabla_xh(tp_{t,i})}{h(tp_{t,i})}+2\alpha_1\frac{p_{t,i}-\underline{e}}{|p_{t,i}-\underline{e}|^2}
+2\alpha_2\frac{p_{t,i}+\underline{e}}{|p_{t,i}+\underline{e}|^2}-\sum_{j\neq i}\frac{\rho_{t,j}}{2\pi}\frac{p_{t,i}-p_{t,j}}{|p_{t,i}-p_{t,j}|^2}\\
=~&t\nabla_xR_t(tp_{t,i})-t\sum_{j=1}^m\rho_{t,j}\nabla_xR(tp_{t,j},tp_{t,i})-t\nabla w(tp_{t,i})+O(t+t\|\t\phi\|_*).
\end{aligned}
\end{equation}
In order to analyze $v_{t}$ in $B_{2r_0}(p_{t,i})$, we set
\begin{align}
\label{s.11}
I_{t,i}(y)=\log\frac{e^{\lambda_{t,i}}}{(1+C_{t,i}e^{\lambda_{t,i}}|y-q_{t,i}|^2)^2},
\end{align}
where $\lambda_{t,i}=v_t(p_{t,i})$, $C_{t,i}=\frac{h_{1,t}(tp_{t,i})}{8}$ and $q_{t,i}$ is chosen such that $|q_{t,i}-p_{t,i}|\ll1$ and
\begin{align*}
\nabla_yI_{t,i}(p_{t,i})=-t\rho_{t,i}\nabla_xR(y_{i,t},y_{i,t})-t\sum_{j\neq i}\rho_{t,j}\nabla_xG(y_{j,t},y_{i,t})-t\nabla w(y_{i,t}).
\end{align*}
By direct computation, we have
\begin{align}
\label{s.12}
|q_{t,i}-p_{t,i}|=O(e^{-\lambda_t}).
\end{align}
For $y\in B_{2r_0}(p_{t,i}),$ we set
\begin{align}
\label{s.13}
\eta_{t,i}(y)=~&v_{t}(y)-I_{t,i}(y)-(G_{t,i}(ty)-G_{t,i}(tp_{t,i})),
\end{align}
It is easy to see that
\begin{align}
\label{s.14}
\eta_{t,i}(p_{t,i})=v_{t}(p_{t,i})-I_{t,i}(p_{t,i})=O(t^2),~\nabla\eta_{t,i}(p_{t,i})=O(t^2).
\end{align}
Let $R_{t,i}=\left(C_{t,i}e^{\lambda_{t,i}}\right)^{\frac12}$ and $\tilde\eta_{t,i}$ be the scaled function of $\eta_{t,i}$, that is
\begin{align*}
\tilde{\eta}_{t,i}(z)=\eta_{t,i}(p_{t,i}+R_{t,i}^{-1}z)~\mathrm{for}~|z|\leq 2R_{t,i}r_0.
\end{align*}
For $\tilde{\eta}_{t,i}(z)$, we have the following estimate
\begin{align}
\label{s.15}
|\tilde{\eta}_{t,i}(z)|\leq C_{\e}(t\|\tilde{\phi}_t\|_*+t^2)(1+|z|)^{\e},~\e\in(0,1).
\end{align}
We shall provide the proof of (\ref{s.15}) in Appendix A, see Lemma \ref{le5.1}. Substituting (\ref{s.15}) into (\ref{s.7}) and by taking derivative with respect to $x$ on both sides, we can obtain
$$\nabla u_t(x)=\nabla w+\sum_{i=1}^m\rho_{t,i}\nabla G(x,y_{i,t})+o(1),$$
and together with (\ref{s.7}), we get $\|\t\phi_t\|_*=o(1).$
\medskip

\noindent Step 2. We derive the estimate on $\rho_{t,i}-8\pi$.
\medskip

With the help of (\ref{s.15}), we can improve the estimate on (\ref{s.18}) and get
\begin{equation}
\label{s.18n}
\int_{B_{2r_0}(p_{t,i})}\left(\xi\cdot\nabla h_{1,t}(ty)\right)e^{v_{t}(y)}\mathrm{d}y=\rho_{t,i}\left(\xi\cdot\frac{\nabla h_{1,t}(tp_{t,i})}{h_{1,t}(tp_{t,i})}\right)+O(t^2\log t).
\end{equation}
As a consequence of (\ref{s.20n}) and (\ref{s.18n}), we have
\begin{equation}
\label{s.20}
\begin{aligned}
&t\frac{\nabla_xh(tp_{t,i})}{h(tp_{t,i})}+2\alpha_1\frac{p_{t,i}-\underline{e}}{|p_{t,i}-\underline{e}|^2}
+2\alpha_2\frac{p_{t,i}+\underline{e}}{|p_{t,i}+\underline{e}|^2}-\sum_{j\neq i}\frac{\rho_{t,j}}{2\pi}\frac{p_{t,i}-p_{t,j}}{|p_{t,i}-p_{t,j}|^2}\\
=~&t\nabla_xR_t(tp_{t,i})-t\sum_{j=1}^m\rho_{t,j}R(tp_{t,j},tp_{t,i})-t\nabla w(tp_{t,i})+O(t^2\log t+t\|\t\phi\|_*).
\end{aligned}
\end{equation}
Now we are in the position to obtain the estimate on $\rho_{t,i}-8\pi$. By the definition of $\rho_{t,i}$, we have
\begin{equation}
\label{s.21}
\begin{aligned}
\rho_{t,i}&=\int_{U_{i,t}^{r_0}}\frac{\rho he^{u_t-G_t^{(2)}}}{\int_Mhe^{u_t-G_t^{(2)}}}
=\int_{B_{r_0}(p_{t,i})}h_{1,t}(ty)e^{v_t}\\
&=\int_{B_{r_0}(p_{t,i})}8C_{t,i}e^{I_{t,i}}+\int_{B_{r_0}(p_{t,i})}8C_{t,i}e^{I_{t,i}}H_{t,i}(y,\eta_{t,i}),
\end{aligned}
\end{equation}
where
\begin{align*}
H_{t,i}(y,\eta_{t,i})=\frac{h(ty)|y-\underline{e}|^{2\alpha_1}|y+\underline{e}|^{2\alpha_2}e^{\psi(ty)-R_t(ty)}}
{h(tp_{t,i})|p_{t,i}-\underline{e}|^{2\alpha_1}|p_{t,i}+\underline{e}|^{2\alpha_2}e^{\psi(tp_{t,i})-R_t(tp_{t,i})}}
e^{\eta_{t,i}+G_{t,i}(ty)-G_{t,i}(tp_{t,i})}-1.
\end{align*}
Since $8C_{t,i}\int_{\mathbb{R}^2}e^{I_{t,i}(y)}\mathrm{d}y=8\pi,$ from (\ref{s.21}), we have
\begin{align}
\label{s.22}
\rho_{t,i}-8\pi=\int_{B_{r_0}(p_{t,i})}8C_{t,i}e^{I_{t,i}}H_{t,i}(y,\eta_{t,i})+O(t^2).
\end{align}
Furthermore, the term $H_{t,i}(y,\eta_{t,i})$ can be expanded as follows:
\begin{equation}
\label{s.23}
\begin{aligned}
H_{t,i}(y,\eta_{t,i})=&~\eta_{t,i}+H_{t,i}(y,0)+H_{t,i}(y,0)\eta_{t,i}+O(1)|\eta_{t,i}|^2\\
=&~\eta_{t,i}(y)+\l b_{t,i},y-p_{t,i}\r+\l B_{t,i}(y-p_{t,i}),(y-p_{t,i})\r\\
&~+O(1)|y-p_{t,i}|^{2+\beta}+H_{t,i}(y,0)\eta_{t,i}+O(1)|\eta_{t,i}|^2,
\end{aligned}
\end{equation}
where $b_{t,i}$ and $B_{t,i}$ are the gradient and Hessian of $H_{t,i}(y,0)$ at $y=p_{t,i}$ and $\beta>0.$ We note that
\begin{align}
\label{s.24}
\Delta H_{t,i}(y,0)\mid_{y=p_{t,i}}=O(t^2).
\end{align}
Using (\ref{s.15}), (\ref{s.20}) and (\ref{s.24}), by direct computations we can get the following estimates:
\begin{align}
\label{s.25}
\left|\int_{B_{r_0}(p_{t,i})}e^{I_{t,i}}\eta_{t,i}(y)\mathrm{d}y\right|=O(t\|\tilde{\phi}_t\|_*+t^2),
\end{align}
\begin{align}
\label{s.26}
\left|\int_{B_{r_0}(p_{t,i})}e^{I_{t,i}}\l b_{t,i},y-p_{t,i}\r\mathrm{d}y\right|+\left|\int_{B_{r_0}(p_{t,i})}e^{I_{t,i}}\l B_{t,i}(y-p_{t,i},y-p_{t,i})\r\mathrm{d}y\right|=o(t^2),
\end{align}
\begin{align}
\label{s.27}
\left|\int_{B_{r_0}(p_{t,i})}e^{I_{t,i}(y)}(|y-p_{t,i}|^{2+\beta}+\eta_{t,i}^2)\mathrm{d}y\right|
+\left|\int_{B_{r_0}(p_{t,i})}e^{I_{t,i}(y)}H_i(y,0)\eta_{t,i}\mathrm{d}y\right|=O(t^2).
\end{align}
Putting the estimates (\ref{s.25})-(\ref{s.27}) into (\ref{s.22}), we get
\begin{align}
\label{s.28}
\rho_{t,i}-8\pi=O(t\|\tilde{\phi}\|_*+t^2\log t).
\end{align}
A direct consequence of (\ref{s.20}) and (\ref{s.28}) is that, for any $i\in\{1,2,\cdots,m\}$, the blow up points $\{p_1,p_2,\cdots,p_m\}$ satisfy the following $m$-equations:
\begin{align}
\label{s.29}
\alpha_1\frac{p_{i}-\underline{e}}{|p_{i}-\underline{e}|^2}
+\alpha_2\frac{p_{i}+\underline{e}}{|p_{i}+\underline{e}|^2}-2\sum_{j\neq i}\frac{p_{i}-p_{j}}{|p_{i}-p_{j}|^2}=0,~i=1,2,\cdots,m.
\end{align}
In section 6, we shall prove that (\ref{s.29}) admits a unique solution.
\medskip

\noindent Step 3. We establish the estimate of $\|\t\phi_t\|_*$.
\medskip

We can represent $\rho$ as the following:
\begin{align}
\label{s.30}
\rho=~&\sum_{i=1}^m\frac{\int_{U_{i,t}^{2r_0}}\rho he^{u_t-G_t^{(2)}}}{\int_Mhe^{u_t-G_t^{(2)}}}+
\frac{\int_{M\setminus\bigcup_{i=1}^mU_{i,t}^{2r_0}}\rho he^{u_t-G_t^{(2)}}}{\int_Mhe^{u_t-G_t^{(2)}}}\nonumber\\
=~&8m\pi+O(t\|\t\phi_t\|_*+t^2\log t)+\frac{\int_{M\setminus\bigcup_{i=1}^mU_{i,t}^{2r_0}}\rho he^{u_t-G_t^{(2)}}}{\int_M he^{u_t-G_t^{(2)}}},
\end{align}
where we used that $he^{u_t-G_t^{(2)}}$ is bounded in $M\setminus\bigcup_{i=1}^mU_{i,t}^{r_0}.$

In $M\setminus\bigcup_{i=1}^mU_{i,t}^{2r_0},$ by (\ref{s.28}), we have
\begin{align}
\label{s.31}
u_t-G_t^{(2)}=~&w+\sum_{i=1}^m8\pi G(x,y_{i,t})+O(t\log t\|\tilde{\phi}_t\|_*+t^2(\log t)^2)-G_t^{(2)}+\tilde{\phi}_t\nonumber\\
=~&w_m+\t\phi_t+\log\left(\frac{|x|^{4m}|x-t\underline{e}|^{2\alpha_1}|x+t\underline{e}|^{2\alpha_2}}{|x|^{2\alpha_1+2\alpha_2}\prod_{i=1}^m|x-tp_{t,i}|^4}\right)
+O(t\log t\|\t\phi_t\|_*+t) \nonumber\\
=~&w_m+\t\phi_t+O\left(\frac{t}{|x|}1_{B_{r_0}(0)}\right)+O(t\log t\|\t\phi_t\|_*+\|\t\phi_t\|_*^2+t).
\end{align}
where $1_{B_{r_0}(0)}$ is the characteristic function.

Substituting this expansion into the integral over $M\setminus\bigcup_{i=1}^mU_{i,t}^{2r_0}$ we get,
\begin{equation}
\label{s.32}
\begin{aligned}
&\int_{M\setminus\bigcup_{i=1}^mU_{i,t}^{2r_0}}he^{u_t-G_t^{(2)}}\\
=~&\int_{M\setminus\bigcup_{i=1}^mU_{i,t}^{2r_0}}he^{w_m}\left(1+\t\phi_t+O(\frac{t}{|x|}1_{B_{r_0}(0)}+t+t\log t\|\t\phi_t\|_*+\|\t\phi_t\|_*^2) \right)\\
=~&\int_Mhe^{w_m}+\int_{M\setminus\bigcup_{i=1}^mU_{i,t}^{2r_0}}he^{w_m}\t\phi_t+O(t+t\log t\|\t\phi_t\|_*+\|\t\phi_t\|_*^2).
\end{aligned}
\end{equation}
Using (\ref{s.30})-(\ref{s.32}), we find,
\begin{equation}
\label{s.33}
\begin{aligned}
\int_Mhe^{u_t-G_t^{(2)}}=~&\frac{\rho}{\rho-8m\pi}\int_Mhe^{w_m}
+\frac{\rho}{\rho-8m\pi}\int_{M\setminus\bigcup_{i=1}^mU_{i,t}^{2r_0}}he^{w_m}\t\phi_t
\\&+O(t+t\log t\|\t\phi_t\|_*+\|\t\phi_t\|_*^2).
\end{aligned}
\end{equation}
From (\ref{s.31}) and (\ref{s.33}), we get the following equation for $\t\phi_t$ in $M\setminus\bigcup_{i=1}^mU_{i,t}^{2r_0}$
\begin{align}
\label{s.34}
\Delta\tilde\phi_t+(\rho-8m\pi)\frac{he^{w_m}}{\int_Mhe^{w_m}}\t\phi_t
-(\rho-8m\pi)\frac{(\int_{M\setminus\bigcup_{i=1}^mU_{i,t}^{4r_0}}he^{w_m}\t\phi_t)he^{w_m}}{(\int_Mhe^{w_m})^2}
=\Pi_1,
\end{align}
where
\begin{align*}
\Pi_1=O\left(\frac{t}{|x|}1_{B_{r_0}(0)}+t+t\log t\|\t\phi_t\|_*+\|\t\phi_t\|_*^2\right).
\end{align*}

In the following, we shall extend the equation (\ref{s.34}) to be defined on $M$. To this purpose we fix a point $x_{t}\in\bigcup_i\partial U_{i,t}^{3r_0}$ and in each set $U_{i,t}^{4r_0}$ we let
\begin{align*}
\tilde{\phi}_t^{ext}=\chi_{t}(x)\tilde{\phi}_t(x)+(1-\chi_{t}(x))\tilde{\phi}_t(x_{t}),
\end{align*}
where
\begin{align*}
\chi_{t}(x)=\left\{\begin{array}{ll}
1,~&\mathrm{for}~x\in M\setminus(\bigcup_{i=1}^m U_{i,t}^{4r_0}),\\
0,~&\mathrm{for}~x\in \bigcup_{i=1}^mU_{i,t}^{2r_0}.
\end{array}\right.
\end{align*}
Then we can see that $\tilde{\phi}_t^{ext}$ satisfies
\begin{align}
\label{s.35}
\Delta\tilde{\phi}_t^{ext}+(\rho-8m\pi)\frac{he^{w_m}}{\int_Mhe^{w_m}}\tilde{\phi}_t^{ext}-
(\rho-8m\pi)\frac{(\int_Mhe^{w_m}\tilde{\phi}_t^{ext})he^{w_m}}{(\int_Mhe^{w_m})^2}=\Pi_2~\mathrm{in}~M,
\end{align}
with
\begin{align*}
\Pi_2=~&(\Delta\chi_{t})(\tilde{\phi}_t(x)-\tilde{\phi}_t(x_{t}))
+2\nabla\chi_{t}\nabla\tilde{\phi}_t+\frac{he^{w_m}(1-\chi_t)\t\phi_t(x_t)}{\int_Mhe^{w_m}}\nonumber\\
&-\sum_{i=1}^m(\rho-8m\pi)\frac{(\int_{U_{i,t}^{4r_0}}he^{w_m}\tilde{\phi}^{ext}_t)he^{w_m}}{(\int_Mhe^{w_m})^2}
\nonumber\\
&+O(\frac{t}{|x|}1_{B_{r_0}(0)}+t+t\log t\|\tilde{\phi}_t\|_*+\|\tilde{\phi}_t\|_*^2).
\end{align*}
It is easy to see that,
\begin{align*}
\|\Pi_2\|_{L^{\delta}(M)}=O(t+t^{\frac{2}{\delta}-1}\|\tilde{\phi}_t\|_*+\|\tilde{\phi}_t\|_*^2),
\end{align*}
with suitable $\delta\in(1,2).$ Since $w$ is non-degenerate, we obtain:
\begin{align}
\label{s.36}
\left\|\tilde{\phi}_t^{ext}-\int_M\tilde{\phi}_t^{ext}\right\|_{W^{2,\delta}(M)}\leq C(t+t^{\frac{2}{\delta}-1}\|\tilde\phi_t\|_*+\|\tilde\phi_t\|_*^2).
\end{align}
It is not difficult to see that $\int_M\tilde{\phi}_t^{ext}=O(t^2)\|\tilde\phi_t\|_*+O(t^2\log t)$. Hence, by Sobolev embedding and (\ref{s.36}), we have
\begin{align}
\label{s.37}
\|\tilde{\phi}_t^{ext}\|_{L^\infty(M)}\leq C\left(t+t^{\frac{2}{\delta}-1}\|\tilde{\phi}_t\|_*+\|\tilde{\phi}_t\|_*^2\right),
\end{align}
and consequently,
\begin{align}
\label{s.38}
\|\tilde{\phi}_t\|_{L^{\infty}(M\setminus\bigcup_{i=1}^mU_{i,t}^{4r_0})}\leq C\left(t+t^{\frac{2}{\delta}-1}\|\tilde{\phi}_t\|_*+\|\tilde{\phi}_t\|_*^2\right).
\end{align}
Hence, for $x\in U_{i,t}^{4r_0}\setminus U_{i,t}^{2r_0},$ using (\ref{s.38}) we obtain:
\begin{align*}
|\t\phi_t|=|\t\phi_t^{ext}+(1-\chi_{t})(\t{\phi}_t-\t\phi_t(x_{t})|\leq |\t\phi_t^{ext}|+t\|\tilde\phi_t\|_*,
\end{align*}
and it implies from (\ref{s.37}) that:
\begin{align}
\label{s.39}
\|\tilde{\phi}_t\|_{L^{\infty}(M\setminus\bigcup_{i=1}^mU_{i,t}^{2r_0})}\leq C\left(t+t^{\frac{2}{\delta}-1}\|\tilde{\phi}_t\|_*+\|\tilde{\phi}_t\|_*^2\right).
\end{align}
Before studying the term $\nabla\t\phi_t$, we establish the following fact, for $x\in M\setminus(\bigcup_{i=1}^mU_{i,t}^{2r_0}),$
\begin{equation}
\label{s.40}
\begin{aligned}
&\int_{U_{i,t}^{r_0}}\left(\frac{(x-z)_j}{|x-z|^2}
-\frac{(x-tp_{t,i})_j}{|x-tp_{t,i}|^2}\right)\frac{he^{u_t-G_t^{(2)}}}{\int_Mhe^{u_t-G_t^{(2)}}}\mathrm{d}z\\
=~&\int_{B_{r_0}(p_{t,i})}8C_{t,i}e^{I_{t,i}}\left(\frac{(x-ty)_j}{|x-ty|^2}-\frac{(x-tp_{t,i})_j}{|x-tp_{t,i}|^2}\right)\mathrm{d}y\\
&+\int_{B_{r_0}(p_{t,i})}8C_{t,i}e^{I_{t,i}}\left(\frac{(x-ty)_j}{|x-ty|^2}-\frac{(x-tp_{t,i})_j}{|x-tp_{t,i}|^2}\right)O(|\eta_{t,i}|+|y-p_{t,i}|)\mathrm{d}y\\
=~&O(t\|\t\phi_t\|_*+t\log t),
\end{aligned}
\end{equation}
where we used symmetry, the estimate (\ref{s.15}) and the expansion (\ref{s.23}). Using (\ref{s.39}), (\ref{s.40}) and the Green's representation formula we get that,
\begin{equation}
\label{s.41}
\begin{aligned}
\nabla u_t(x)=~&\left(\int_{M\setminus\bigcup_{i=1}^mU_{i,t}^{4r_0}}
+\sum_{i=1}^m\int_{B_{r_0t}(tp_{t,i})}\right)
\nabla G(x,z)\frac{\rho he^{u_t-G_t^{(2)}}}{\int_Mhe^{u_t-G_t^{(2)}}}+O(t)\\
=~&\int_{M\setminus\bigcup_{i=1}^mU_{i,t}^{4r_0}}\nabla G(x,z)\frac{(\rho-8m\pi)he^{w_m}}{\int_Mhe^{w_m}}(1+O(\|\t\phi_t\|_{L^{\infty}(M\setminus \bigcup_{i=1}^mU_{i,t}^{2r_0})}))\\
&+\sum_{i=1}^m\int_{U_{i,t}^{r_0}}\nabla G(x,z)\frac{\rho he^{u_t-G_t^{(2)}}}{\int_Mhe^{u_t-G_t^{(2)}}}+O(t\log t)\\
=~&\int_{M}\nabla G(x,z)\frac{(\rho-8m\pi)he^{w_m}}{\int_Mhe^{w_m}}+\sum_{i=1}^m\rho_{t,i}\nabla G(x,y_{i,t})\\
&+O(t\log t+t^{\frac{2}{\delta}-1}\|\t\phi_t\|_*+\|\t\phi_t\|_*^2)\\
=~&\nabla w(x)+\sum_{i=1}^m\rho_{t,i}\nabla G(x,y_{i,t})+O(t\log t+t^{\frac{2}{\delta}-1}\|\t\phi_t\|_*+\|\t\phi_t\|_*^2)
\end{aligned}
\end{equation}
for $x\in M\setminus\bigcup_{i=1}^mU_{i,t}^{2r_0}$. Therefore, from (\ref{s.39}) and (\ref{s.41}) we can derive that,
\begin{align}
\label{s.42}
\|\tilde{\phi}_t\|_*=O(t\log t).
\end{align}
Combining (\ref{s.28}) and (\ref{s.42}), we can finally obtain the estimate
$$\rho_{t,i}-8\pi=O(t^2\log t),$$
and complete the proof.
\end{proof}

\noindent {\bf Remark:} If $\alpha_1=\alpha_2=1$, then $m=1$. By (\ref{s.20n}), we have $|p_{t,1}|=O(t)$. While the estimate in (\ref{s.31}) can be easily improved by
\begin{equation*}
u_t-G_t^{(2)}=w_m+\tilde\phi_t+O\left(\frac{t^2}{|x|^2}1_{B_{r_0}(0)}\right)+O(t\log t\|\tilde\phi\|_*+t^2\log t).
\end{equation*}
As a consequence, we can follow the arguments in the proof of Proposition \ref{pr3.2} to obtain that,
$$\|u_t-w-\rho_{t,1}G(x,y_{1,t})\|_{L^{\infty}(M\setminus U_{1,t}^{2r_0})}\leq Ct^{2\tau},~\mathrm{for~any}~\tau\in(0,1).$$
This remark is important in the forthcoming work \cite{lly}.
\medskip

Next, we shall derive a relation between $\lambda_{t,i}$ and $\log t$.
\begin{proposition}
\label{pr3.3}
\begin{equation}
\label{r.1}
\begin{aligned}
&\lambda_{t,i}+(2+2\alpha_1+2\alpha_2-4m)\log t\\
=~&\log\left(\frac{\rho}{\rho-8m\pi}\int_Mhe^{w_m}\right)-2\log C_{t,i}
+\sum_{j\neq i}4\log|p_{t,j}-p_{t,i}|\\
&-\sum_{j=1}^m8\pi R(y_{i,t},y_{j,t})-w(tp_i)+O(t\log t)
\end{aligned}
\end{equation}
where
\begin{align*}
w_m=w+4\pi(2m-\alpha_1-\alpha_2)G(x,0).
\end{align*}
\end{proposition}
%and
%\begin{align*}
%C_{0,i}=\frac18\rho h(tp_{i})|p_{i}-e|^{2\alpha_1}|p_{i}+e|^{2\alpha_2}e^{-R_t(tp_{i})}.
%\end{align*}
%

\begin{proof}
From the results above, we see that $\|\tilde{\phi}_t\|_*=O(t\log t)$ and $\rho_{t,i}-8\pi=O(t^2\log t)$. As a consequence, for $x\in\partial U_{i,t}^{2r_0}$ we have:
\begin{align}
\label{r.2}
u_t(x)-w(x)-\sum_{i=1}^m\rho_{t,i}G(x,y_{i,t})=O(t\log t).
\end{align}
On the other hand, we recall that
\begin{align}
\label{r.3}
\eta_{t,i}(y)=~&u_t(ty)-\log\int_Mhe^{u_t-G_t^{(2)}}+(2+2\alpha_1+2\alpha_2)\log t-I_{t,i}(y)\nonumber\\
&-(G_{t,i}(ty)-G_{t,i}(tp_{t,i}))-\psi_{t}.
\end{align}
With the estimate on $\|\t\phi_t\|_*$ and (\ref{s.15}), for $|z|\leq 2R_{t,i}r_0$ we derive that,
\begin{align}
\label{r.4}
|\eta_{t,i}(p_{t,i}+R_{t,i}^{-1}z)|\leq C_{\e}t^2\log t(1+|z|)^{\e},~\e\in(0,1).
\end{align}
From (\ref{s.33}) and the estimate on $\t\phi_t,$ we have
\begin{align}
\label{r.5}
\int_Mhe^{u_t-G_t^{(2)}}=\frac{\rho}{\rho-8m\pi}\int_Mhe^{w_m}+O(t).
\end{align}
Using (\ref{r.3}-\ref{r.5}), for $y\in\partial B_{2r_0}(p_{t,i})$ we deduce that,
\begin{equation}
\begin{aligned}
\label{r.6}
u_t(ty)=~&\log\left(\frac{\rho}{\rho-8m\pi}\int_Mhe^{w_m}\right)-(2\alpha_1+2\alpha_2-2)\log t-\lambda_{t,i}-2\log C_{t,i}\\
&+(\frac{\rho_{t,i}}{2\pi}-4)\log|ty-tp_{t,i}|+\sum_{j=1}^m\rho_{t,j}G(ty,y_{j,t})+w(ty)-G_{t,i}(tp_{t,i})+O(t)\\
=~&\log\left(\frac{\rho}{\rho-8m\pi}\int_Mhe^{w_m}\right)-(2+2\alpha_1+2\alpha_2-4m)\log t-\lambda_{t,i}\\
&-2\log C_{t,i}+\sum_{j\neq i}4\log|p_{t,j}-p_{t,i}|-\sum_{j=1}^m8\pi R(y_{i,t},y_{j,t})\\
&+\sum_{j=1}^m\rho_{t,j}G(ty,y_{j,t})+w(ty)-w(tp_{t,i})+O(t).
\end{aligned}
\end{equation}
Combining (\ref{r.2}) and (\ref{r.6}), we conclude that there holds:
\begin{equation}
\label{r.7}
\begin{aligned}
&\lambda_{t,i}+(2+2\alpha_1+2\alpha_2-4m)\log t\\
=~&\log\left(\frac{\rho}{\rho-8m\pi}\int_Mhe^{w_m}\right)-2\log C_{t,i}
+\sum_{j\neq i}4\log|p_{t,j}-p_{t,i}|\\
&-\sum_{j=1}^m8\pi R(y_{i,t},y_{j,t})-w(tp_{t,i})+O(t\log t),
\end{aligned}
\end{equation}
as claimed.
\end{proof}

\noindent {\em Proof of Theorem \ref{th1.4}.} Theorem \ref{th1.4} is a direct consequence of Proposition \ref{pr3.2}, \ref{pr3.3} and (\ref{r.5}). \hfill $\square$

\section{The uniqueness of the blow up points}
In this section, we prove that the location of the blow up points $p_i,i=1,2,\cdots,m$ are uniquely determined, i.e., (\ref{s.29}) with $i=1,\cdots,m$ admits a unique solution. We shall use the complex number $e_i$ to denote the point $p_i$ and assume $\underline{e}=1$ (after a rotation if necessary). We rewrite (\ref{s.29}) as the following
$$\alpha_1\frac{e_l-\underline{e}}{|e_l-\underline{e}|^2}+\alpha_2\frac{e_l+\underline{e}}{|e_l+\underline{e}|^2}=2\sum_{j\neq l}\frac{e_l-e_j}{|e_l-e_j|^2},~1\leq l\leq m.$$
%where $e=(1,0)$ and $m\leq \min\{\alpha_1,\alpha_2\}.$ Let us regard $e_1,e_2,\cdots,e_m$ as complex numbers.
Then, we have
\begin{equation}
\label{bc}
\frac{\alpha_1}{e_l-1}+\frac{\alpha_2}{e_l+1}=2\sum_{j\neq l}\frac{1}{e_l-e_j}.
\end{equation}
Hence,
\begin{equation*}
\alpha_1(e_l+1)+\alpha_2(e_l-1)=2\sum_{j\neq l}\frac{e_l^2-1}{e_l-e_j}
\end{equation*}
Taking the summation w.r.t. $l,$ we have
\begin{equation*}
(\alpha_1+\alpha_2)\sum_{i=1}^me_i+(\alpha_1-\alpha_2)m=\sum_{i\neq j}2(e_i+e_j)=2(m-1)\sum_{i=1}^me_i,
\end{equation*}
and consequently,
\begin{equation*}
\sum_{i=1}^me_i=\frac{(\alpha_2-\alpha_1)m}{(\alpha_1+\alpha_2-2(m-1))}.
\end{equation*}

Before we proceed further, we let $I=\{e_1,\cdots,e_m\},$ $I_l=I\setminus\{e_l\}$ and $I_{l,j}=I\setminus\{e_l,e_j\}$. We introduce the following notation,
$$\left(\begin{matrix}I\\k\end{matrix}\right)=\sum_{j_1<\cdots<j_k}e_{j_1}\cdots e_{j_k},~\mathrm{where}~e_{j_k}\in I~\mathrm{for}~k=1,\cdots,m$$
and
$$\left(\begin{matrix}I_l\\k\end{matrix}\right)=\sum_{j_1<\cdots<j_k}e_{j_1}\cdots e_{j_k},~\mathrm{where}~e_{j_k}\in I_l~\mathrm{for}~k=1,\cdots,m-1.$$
Clearly, we set $\left(\begin{matrix}I\\0\end{matrix}\right)=1$.

For the equation \eqref{bc}, we have
\begin{equation*}
\alpha_1(e_l+1)\prod_{j\neq l}(e_l-e_j)+\alpha_2(e_l-1)\prod_{j\neq l}(e_l-e_j)=2(e_l^2-1)\sum_{j\neq l}\prod_{m\in I_{l,j}}(e_l-e_m).
\end{equation*}
The left hand side of the above equality can be written as follows:
\begin{equation}
\label{eq}
(\alpha_1+\alpha_2)e_l\prod_{j\neq l}(e_l-e_j)+(\alpha_1-\alpha_2)\prod_{j\neq l}(e_l-e_j)=2(e_l^2-1)\sum_{j\neq l}\prod_{m\in I_{j,l}}(e_l-e_m),
\end{equation}
and we notice that
\begin{equation}
\label{l1}
(\alpha_1+\alpha_2)e_l\prod_{j\neq l}(e_l-e_j)=(\alpha_1+\alpha_2)e_l\sum_{k=0}^{m-1}(-1)^k\left(\begin{matrix}I_l\\k\end{matrix}\right)e_l^{m-1-k},
\end{equation}
and
\begin{equation}
\label{l2}
(\alpha_1-\alpha_2)\prod_{j\neq l}(e_l-e_j)=(\alpha_1-\alpha_2)\sum_{k=0}^{m-1}(-1)^k\left(\begin{matrix}I_l\\k\end{matrix}\right)e_l^{m-1-k}.
\end{equation}
Concerning the right hand side of the above equation, we have:
\begin{equation}
\label{l3}
\begin{aligned}
\sum_{j\neq l}\prod_{m\in I_{l,j}}(e_l-e_m)=&~(m-1)e_l^{m-2}+\sum_{k=1}^{m-2}(-1)^k(m-1-k)
\left(\begin{matrix}I_l\\k\end{matrix}\right)e_l^{m-2-k}\\
=&~\sum_{k=0}^{m-2}(-1)^k(m-1-k)\left(\begin{matrix}I_l\\k\end{matrix}\right)e_l^{m-2-k}\\
\end{aligned}
\end{equation}

Since
\begin{equation*}
\left(\begin{matrix}I_l\\k\end{matrix}\right)=\sum_{i=0}^k(-1)^{k-i}\left(\begin{matrix}I\\i\end{matrix}\right)e_l^{k-i},
\end{equation*}
then we can rewrite the equation \eqref{l1} to \eqref{l3} as follows:
\begin{align*}
(\alpha_1+\alpha_2)e_l\prod_{j\neq l}(e_l-e_j)
=&~(\alpha_1+\alpha_2)e_l\sum_{k=0}^{m-1}\sum_{i=0}^k(-1)^{2k-i}\left(\begin{matrix}I\\i\end{matrix}\right)e_l^{m-1-i}\\
=&~(\alpha_1+\alpha_2)e_l\sum_{i=0}^{m-1}\sum_{k=i}^{m-1}(-1)^i\left(\begin{matrix}I\\i\end{matrix}\right)e_l^{m-1-i}\\
=&~(\alpha_1+\alpha_2)e_l\sum_{k=0}^{m-1}(-1)^k(m-k)\left(\begin{matrix}I\\k\end{matrix}\right)e_l^{m-1-k},
\end{align*}
\begin{align*}
(\alpha_1-\alpha_2)\prod_{j\neq l}(e_l-e_j)
&=~(\alpha_1-\alpha_2)\sum_{k=0}^{m-1}\sum_{i=0}^k(-1)^{2k-i}\left(\begin{matrix}I\\i\end{matrix}\right)e_l^{m-1-i}\\
&=~(\alpha_1-\alpha_2)\sum_{k=0}^{m-1}(-1)^k(m-k)\left(\begin{matrix}I\\k\end{matrix}\right)e_l^{m-1-k},
\end{align*}
and
\begin{align*}
\sum_{j\neq l}\prod_{m\in I_{l,j}}(e_l-e_m)&=~\sum_{k=0}^{m-2}\sum_{i=0}^k(-1)^{2k-i}(m-1-k)\left(\begin{matrix}I\\i\end{matrix}\right)e_l^{m-2-i}\\
&=~\sum_{k=0}^{m-2}(-1)^k\frac{(m-k)(m-k-1)}{2}\left(\begin{matrix}I\\k\end{matrix}\right)e_l^{m-2-k}.
\end{align*}
So, we have
\begin{align*}
&(\alpha_1+\alpha_2)\sum_{k=0}^{m-1}(-1)^k(m-k)\left(\begin{matrix}I\\k\end{matrix}\right)e_l^{m-k}+
(\alpha_1-\alpha_2)\sum_{k=0}^{m-1}(-1)^k(m-k)\left(\begin{matrix}I\\k\end{matrix}\right)e_l^{m-1-k}\\
&~\quad=\sum_{k=0}^{m-2}(-1)^k(m-k-1)(m-k)\left(\begin{matrix}I\\k\end{matrix}\right)(e_l^{m-k}-e_l^{m-k-2}).
\end{align*}
Therefore, $e_l$ satisfies
\begin{align*}
(\alpha_1+\alpha_2-(m-1))mz^m+\sum_{k=0}^{m-1}c_kz^{k}=0,
\end{align*}
where
\begin{align*}
c_{m-1}=\left((m-2-(\alpha_1+\alpha_2))(m-1)\left(\begin{matrix}I\\1\end{matrix}\right)+m(\alpha_1-\alpha_2)\right),
\end{align*}
and
\begin{equation}
\begin{aligned}
\label{l4}
c_{m-2-k}=&~(-1)^k\left((\alpha_1+\alpha_2)(m-2-k)-(m-k-3)(m-k-2)\right)\left(\begin{matrix}I\\k+2\end{matrix}\right)\\
&~+(-1)^{k-1}(\alpha_1-\alpha_2)(m-k-1)\left(\begin{matrix}I\\k+1\end{matrix}\right)\\
&~+(-1)^k(m-k)(m-k-1)\left(\begin{matrix}I\\k\end{matrix}\right),\quad~0\leq k\leq m-2.
\end{aligned}
\end{equation}
By Vita formula, we have
\begin{align}
\label{l5}
c_{m-2-k}=(-1)^k\left(\alpha_1+\alpha_2-(m-1)\right)m\left(\begin{matrix}I\\k+2\end{matrix}\right),\quad~0\leq k\leq m-2.
\end{align}
Combining (\ref{l4}) and (\ref{l5}), we have
\begin{align*}
\left(\begin{matrix}I\\k+2\end{matrix}\right)=
\frac{(m-k-1)(m-k)\left(\begin{matrix}I\\k\end{matrix}\right)
-(\alpha_1-\alpha_2)(m-k-1)\left(\begin{matrix}I\\k+1\end{matrix}\right)}
{(2+k)(\alpha_1+\alpha_2-2m+k+3)},
\end{align*}
We note that $\alpha_1+\alpha_2-2m+k+3\neq0$. So $\left(\begin{matrix}I\\k+2\end{matrix}\right)$ is uniquely determined by
$\left(\begin{matrix}I\\k\end{matrix}\right)$ and $\left(\begin{matrix}I\\k+1\end{matrix}\right)$. On the other hand, we already know
$\left(\begin{matrix}I\\0\end{matrix}\right)$ and $\left(\begin{matrix}I\\1\end{matrix}\right)$, hence we can uniquely get
$\left(\begin{matrix}I\\k\end{matrix}\right)$ by induction. Hence, we have proved that $e_l$ corresponds to the zeroes of a given polynomial and thus they are uniquely determined in this way. Therefore, the blow up points $\{p_1,p_2,\cdots,p_m\}$ are uniquely determined as well.

\section{Appendix A: The estimate (\ref{s.15})}
We establish the estimate (\ref{s.15}) by using the following lemma.
\begin{lemma}
\label{le5.1}
For any $\e\in(0,1)$, there exists $C_{\e}$ such that for all $t$ small,
\begin{align*}
|\t\eta_{t,i}(z)|\leq C_{\e}(t\|\t\phi_t\|_*+t^2)(1+|z|)^{\e}.
\end{align*}
\end{lemma}
\begin{proof}
We write the equation for $\t\eta_{t,i}(z)$ as follows:
\begin{align}
\label{5.1}
\left\{\begin{array}{l}
\Delta_z\t\eta_{t,i}(z)+8C_{t,i}e^{\xi_{t,i}}\t\eta_{t,i}(z)=O\left(e^{-\frac{\lambda_t}{2}}t(1+|z|)^{-3}
+e^{-\lambda_t}t^2\sum_{j=1}^m\rho_{t,j}\right),\\
\t\eta_{t,i}=O(t^2),~\nabla_z\t{\eta}_{t,i}(0)=O(t^2e^{-\lambda_t/2}),
\end{array}\right.
\end{align}
where $\xi_{t,i}$ is obtained by the mean value theorem. On $\partial B_{2R_{t,i}r_0}(0)$, by (\ref{s.10}) and (\ref{s.13}), we have
\begin{equation}
\label{5.2}
\begin{aligned}
\t\eta_{t,i}=~&(4-\frac{\rho_{t,i}}{2\pi})\log 2r_0+2\log C_{t,i}+\lambda_{t,i}+G_{t,i}(tp_{t,i})-\log\int_Mhe^{u_t-G_t^{(2)}}\\
&+(2+2\alpha_1+2\alpha_2-\frac{\rho_{t,i}}{2\pi})\log t+\tilde{\phi}_t(t(p_{t,i}+R_{t,i}^{-1}z))-\psi_t(p_{t,i}+R_{t,i}^{-1}z)+O(t^2)\\
=~&\Theta_{t,i}+O(t^2+t\|\tilde{\phi}_t\|_*),
\end{aligned}
\end{equation}
where
\begin{align*}
\Theta_{t,i}=~&(4-\frac{\rho_{t,i}}{2\pi})\log 2r_0+2\log C_{t,i}+\lambda_{t,i}+G_{t,i}(tp_{t,i})-\log\int_Mhe^{u_t-G_t^{(2)}}\nonumber\\
&+(2+2\alpha_1+2\alpha_2-\frac{\rho_{t,i}}{2\pi})\log t+\tilde{\phi}_t(t(p_{t,i}+r_{0,i}))-\psi_t(p_{t,i}+r_{0,i})
\end{align*}
for some $r_{0,i}\in\partial B_{2r_0}(0).$ Let
\begin{align*}
\Lambda_{t,i}=\max\frac{|\t\eta_{t,i}(z)|}{(t\|\t\phi_t\|_*+t^{2})(1+|z|)^{\e}},~z\in B_{2R_{t,i}r_0}(0).
\end{align*}
The statement of the lemma follows once we show that, $\Lambda_{t,i}=O(1).$ We prove this by contradiction. Suppose $\Lambda_{t,i}\rightarrow\infty$, then we use $z_{t,i}$ to denote the point where $\Lambda_{t,i}$ is assumed. Let
\begin{align*}
\bar{\eta}_{t,i}(z)=\frac{\t\eta_{t,i}(z)}{\Lambda_{t,i}(t\|\t\phi_t\|_*+t^{2})(1+|z_{t,i}|)^{\e}}.
\end{align*}
It is straightforward to see that $|\bar\eta_{t,i}(z)|\leq\frac{(1+|z|)^{\e}}{(1+|z_{t,i}|)^{\e}},$ which means that $\bar\eta_{t,i}(z)$ is uniformly bounded over any fixed compact subset of $\mathbb{R}^2.$ Therefore, we conclude that, along a subsequence, $\bar\eta_{t,i}$ converges in $C_{\mathrm{loc}}^2(\R^2)$ to a solution of the following problem:
\begin{align*}
\left\{\begin{array}{l}
\Delta\bar\eta_{0,i}+\frac{C_{0,i}}{(1+|z|^2)^2}\bar\eta_{0,i}=0,\\
\bar\eta_{0,i}(0)=\nabla\bar\eta_{0,i}(0)=0,~|\bar\eta_{0,i}(z)|\leq C(1+|z|)^{\e}.
\end{array}\right.
\end{align*}
If we suppose $z_{t,i}$ converges to $z_{0,i}\in\mathbb{R}^2$, then we have $\max|\bar\eta_{0,i}|=|\bar\eta_{0,i}(z_{0,i})|=1$ by continuity. But this is impossible, since the only bounded solution of the problem above is given by $\bar\eta_{0,i}\equiv0$ (see   \cite[Proposition 1]{bp}). Therefore, $z_{t,i}\rightarrow\infty.$
\medskip

By using $|\bar\eta_{t,i}(z_{t,i})|=1$, together with the boundary data (\ref{5.2}) and the Green representation formula, we have:
\begin{align}
\label{5.3}
\pm1=&\int_{B_{2R_{t,i}r_0}(0)}G_i(z_{t,i},z)
\Big[\frac{8C_{t,i}e^{\xi_{t,i}}\t\eta_{t,i}(z)(1+|z|)^{-\e}}{\Lambda_{t,i}(t\|\t{\phi}\|_*+t^{2})}\frac{(1+|z|)^{\e}}{(1+|z_{t,i}|)^{\e}}
+\frac{O(1)(1+|z|^{-3})}{\Lambda_{t,i}(1+|z_{t,i}|)^{\e}}\Big]\nonumber\\
&+\int_{B_{2R_{t,i}r_0}(0)}G_i(z_{t,i},z)\frac{O(1)(t^2e^{-\lambda_{t,i}}\sum_{j=1}^m\rho_{t,j})}
{\Lambda_{t,i}(t\|\t{\phi}\|_*+t^{2})(1+|z_{t,i}|)^{\e}}\\
&-\int_{\p B_{2R_{t,i}r_0}(0)}\p_{\nu}G_i(z_{t,i},z)\left(\frac{\Theta_{t,i}+O(t\|\t\phi_t\|_*+t^2)}
{\Lambda_{t,i}(t\|\t\phi_t\|_*+t^{2})(1+|z_{t,i}|)^{\e}}\right)\nonumber,
\end{align}
where $G_i$ is the Green's function over $B_{2R_{t,i}r_0}(0)$  with respect to the Dirichlet boundary condition. Recall that
\begin{align*}
G_i(a,z)=-\frac{1}{2\pi}\log|a-z|+\frac{1}{2\pi}\log\left(\frac{|a|}{2R_{t,i}r_0}
\left|\frac{(2R_{t,i}r_0)^2a}{|a|^2}-z\right|\right).
\end{align*}
By the definition it is not difficult to see that $\bar\eta_{t,i}(0)=o(1)$. Applying the Green's representation formula again, we have
\begin{align}
\label{5.4}
o(1)=&\int_{B_{2R_{t,i}r_0}(0)}G_i(0,z)
\Big[\frac{8C_{t,i}e^{\xi_{t,i}}\t\eta_{t,i}(z)(1+|z|)^{-\e}}{\Lambda_{t,i}(t\|\t{\phi}\|_*+t^{2})}\frac{(1+|z|)^{\e}}{(1+|z_{t,i}|)^{\e}}
+\frac{O(1)(1+|z|^{-3})}{\Lambda_{t,i}(1+|z_{t,i}|)^{\e}}\Big]\nonumber\\
&+\int_{B_{2R_{t,i}r_0}(0)}G_i(0,z)\frac{O(1)(t^2e^{-\lambda_{t,i}}\sum_{j=1}^m\rho_{t,j})}
{\Lambda_{t,i}(t\|\t{\phi}\|_*+t^{2})(1+|z_{t,i}|)^{\e}}\\
&-\int_{\p B_{2R_{t,i}r_0}(0)}\p_{\nu}G_i(0,z)\left(\frac{\Theta_{t,i}+O(t\|\t\phi_t\|_*+t^2)}
{\Lambda_{t,i}(t\|\t\phi_t\|_*+t^{2})(1+|z_{t,i}|)^{\e}}\right)\nonumber.
\end{align}
To deal with the two boundary integral terms in (\ref{5.3}) and (\ref{5.4}), we note that
$$\int_{\p B_{2R_{t,i}r_0}(0)}\p_{\nu}G_i(a,z)\mathrm{d}s=-1,~\forall a\in B_{2R_{t,i}r_0}.$$
As a consequence, we can show that the difference of the boundary integral terms is small. At the same time, using that both $\Lambda_{t,i}$ and $|z_{t,i}|\rightarrow+\infty$, we can easily see that the second term on the right hand side of (\ref{5.3}) and (\ref{5.4}) are both very small. Therefore, from (\ref{5.3}) and (\ref{5.4}), we can obtain
\begin{align}
\label{5.5}
1+o(1)=\int_{B_{2R_{t,i}r_0}(0)}|G_i(z_{t,i},z)-G_i(0,z)|\left(\frac{(1+|z|)^{-4+\e}}{(1+|z_{t,i}|)^{\e}}
+o(1)\frac{(1+|z|)^{-3+\e}}{(1+|z_{t,i}|)^{\e}}\right),
\end{align}
where we used that
\begin{align*}
\left|\frac{\t\eta_{t,i}(z)(1+|z|)^{\e}}{\Lambda_{t,i}(t\|\t\phi_t\|_*+t^{2})}\right|\leq 1,~e^{\xi_{t,i}(z)}\leq C(1+|z|)^{-4},~\Lambda_{t,i}\rightarrow\infty.
\end{align*}
To get a contradiction to (\ref{5.5}) we then show that the right hand side is $o(1)$. We consider two cases: If $|z_{t,i}|=o(1)R_{t,i}$,  then $G_i(z_{t,i},z)$ can be written as follows:
\begin{align*}
G_i(z_{t,i},z)=-\frac{1}{2\pi}\log|z_{t,i}-z|+\frac{1}{2\pi}\log(2R_{t,i}r_0)+o(1).
\end{align*}
In this case it is enough to observe that
\begin{align*}
\int_{B_{2R_{t,i}r_0}(0)}\left|\log\frac{|z_{t,i}-z|}{|z|}\right|
\left(\frac{(1+|z|)^{-4+\e}}{(1+|z_{t,i}|)^{\e}}+o(1)\frac{(1+|z|)^{-3+\e}}{(1+|z_{t,i}|)^{\e}}\right)\mathrm{d}z=o(1),
\end{align*}
which follows by standard elementary estimates.

Finally we need to consider the case $|z_{t,i}|\sim R_{t,i}$. Then, for the Green's function we use that
\begin{align*}
|G_i(z_{t,i},z)-G_i(0,z)|\leq C(\log(1+|z|)+\lambda_{t,i}),
\end{align*}
and so it is easy to obtain that the right hand side of (\ref{5.5}) is $o(1).$ Thus we have reached a contradiction and the proof of the Lemma \ref{le5.1} is completed.
\end{proof}

\section{Appendix B: The proof of Proposition \ref{pr2.new}}
In this section, we shall provide a complete proof of Proposition \ref{pr2.new}.
\medskip

\noindent{\em Proof of Proposition \ref{pr2.new}.} Part (i) about  the radial symmetry of the solution has been recently established (in a more general context) by Gui-Moradifam in \cite[Theorem 5.2]{gm}. While, the existence, uniqueness and non-degeneracy properties of the radial solutions, when $\rho\in(8\pi\max\{1,\alpha\},8\pi(\alpha+1))$, have been shown by the second author in \cite{l0}, whose proof extends also when $\rho=8\pi\alpha$ and $\alpha>1.$ Thus, we are left to show (iii) for $\alpha>1$. To this purpose we recall that, if $\alpha>1$ and (\ref{new13}) admits a solution then necessarily,
\begin{equation}
\label{n8.1}
4\pi(\alpha+1)<\rho<8\pi(\alpha+1).
\end{equation}
Inequality (\ref{n8.1}) follows by the integrability condition and a Pohozaev type identity obtained in the usual way, see \cite{l0} for details.

To proceed further, we follow \cite{l0} and consider the Cauchy problem associated to \eqref{new13}, namely:
\begin{equation}
\label{n8.2}
\left\{\begin{array}{l}
-(rv'(r))'=r(1+r^2)^\alpha e^{v(r)},\\
v(0)=a,\quad v'(0)=0
\end{array}\right.
\end{equation}
with $\alpha>-1$ and $a\in\mathbb{R}.$ From \cite{l0}, we know that problem (\ref{n8.2}) admits a unique solution $v_{\alpha}(r,a)$ defined for any $r\geq0$ and such that:
\begin{equation}
\label{n8.3}
\beta_{\alpha}(a):=\int_0^{\infty}(1+r^2)^{\alpha}e^{v_{\alpha}(r,a)}r\mathrm{d}r<+\infty.
\end{equation}
So, the full range of $\rho$ for which (\ref{new13}) admits a radial solution is described by the image of the function: $2\pi\beta_{\alpha}(a),a\in\mathbb{R}.$ Clearly $\beta_{\alpha}(a)$ depends continuously on both parameters $(\alpha,a)\in(-1,+\infty)\times\mathbb{R},$ and actually for fixed $\alpha>-1$, we easily check that $\beta_{\alpha}\in C^2(\mathbb{R}),$ see \cite{l0}.

From (\ref{n8.1}), we know in particular that
\begin{equation}
\label{n8.4}
\mathrm{if}~\alpha>1~\mathrm{then}~\beta_{\alpha}(a)\in(2(\alpha+1),4(\alpha+1)).
\end{equation}
Furthermore, we recall that $\beta_{\alpha}(a)$ admits a limit as $a\to\pm\infty,$ and the following holds:
\begin{equation}
\label{n8.5}
\lim_{a\to-\infty}\beta_{\alpha}(a)=4(\alpha+1)~\mathrm{and}~\lim_{a\to+\infty}\beta_{\alpha}(a)=4\max\{\alpha,1\}.
\end{equation}
Properties (\ref{n8.5}) have been established in \cite[Lemma 2.2]{l0}. For later use we show that (\ref{n8.5}) follow by a suitable blow-up (blow-down) argument. Indeed, by setting
\begin{equation*}
\lambda_{a}=e^{-\frac{a}{2(\alpha+1)}}\to+\infty~\mathrm{as}~a\to-\infty.
\end{equation*}
We easily get that the scaled (blow-down) function:
\begin{equation*}
V_a(r):=v_{\alpha}(\lambda_ar,a)-a\to\log\left(\frac{1}{(1+\frac{1}{8(\alpha+1)^2}r^{2(\alpha+1)})^2}\right)
\end{equation*}
as $a\to-\infty$, uniformly in $C_{\mathrm{loc}}^2.$ This information together with (\ref{n8.4}), suffices to deduce that, $$\lim_{a\to-\infty}\beta_{\alpha}(a)=4(\alpha+1).$$
To show that $\lim_{a\to+\infty}\beta_{\alpha}(a)=4\alpha$ is more delicate. Indeed, while we easily see that $v_{\alpha}(r,a)$ admits a blow up point at the origin, as $a\to+\infty$, in order to determine the limit of $\beta_{\alpha}(a),$ we need to control the behavior of $v_{\alpha}(r,a)$ at infinity. To this end, we consider:
\begin{equation}
\label{n8.6}
\hat v_{\alpha}(r,a)=v_{\alpha}(\frac1r,a)+\beta_{\alpha}(a)\log\frac{1}{r},
\end{equation}
which extends smoothly at $r=0$ and satisfies:
\begin{equation}
\label{n8.7}
\begin{cases}
-(r\hat v_{\alpha}'(r,a))'=r^{\beta_{\alpha}(a)-2(\alpha+2)+1}(1+r^2)^{\alpha}e^{\hat v_{\alpha}(r,a)},\\
\beta_{\alpha}(a)=\int_0^{\infty}r^{\beta_{\alpha}(a)-2(\alpha+2)+1}(1+r^2)^{\alpha}e^{\hat v_{\alpha}(r,a)}\mathrm{d}r.
\end{cases}
\end{equation}
Notice that the blow up analysis of \cite{bt1,bm,ls}, can be applied to both $v_{\alpha}(r,a)$ and $\hat v_{\alpha}(r,a),$ as $a\to+\infty$ and implies the following:
\begin{itemize}
\item if $-1<\alpha\leq1$ then $\hat v_{\alpha}(r,a)$ cannot blow up (at the origin) as $a\to+\infty$, and $\lim_{a\to+\infty}\beta_{\alpha}(a)=4;$
\item if $\alpha>1$ then both $v_{\alpha}(r,a)$ and $\hat v_{\alpha}(r,a)$ must blow up at the origin (and only at the origin). More precisely, along a sequence $a_n\to+\infty$, the sequence $v_{\alpha}(r,a_n)$ blows up only at the origin with blow up mass equal to $8\pi$, and $\hat v_{\alpha}(r,a_n)$ blows up only at the origin with blow-up mass equal to $2\pi(2\beta_{\infty}-4(\alpha+1)),$ where $\beta_{\infty}=\lim_{n\to+\infty}\beta_{\alpha}(a_n)$. As a consequence, by the "concentration" property, for $\beta_{\infty}$ we obtain the identity: $2\pi\beta_{\infty}=8\pi+2\pi(2\beta_{\infty}-4(\alpha+1))$ that gives: $\beta_{\infty}=4\alpha.$ Since this holds along any $a_n\to+\infty$, as $n\to+\infty$, we conclude that
    $$\beta_{\alpha}(a)\to4\alpha,~\mathrm{as}~a\to+\infty,$$
    as claimed.
\end{itemize}

\noindent {\bf Remark 8.1}: For later use, notice that the blow up property at the origin remains valid also for the sequences:
$$v_{\alpha_n}(r,a_n)\quad\mathrm{and}\quad\hat v_{\alpha_n}(r,a_n).$$
When
$$\alpha_n\to\alpha>1\quad\mathrm{and}\quad a_n\to+\infty,~\mathrm{as}~n\to\infty,$$
and it implies that
$$\beta_{\alpha_n}(a_n)\to4\alpha,~\mathrm{as}~n\to+\infty.$$

To proceed further, let
$$\varphi_\alpha(r,a)=\frac{\partial}{\partial a}v_{\alpha}(r,a),$$
which defines a solution of the following linearized problem around $v_\alpha(r,a):$
\begin{equation}
\label{n8.8}
\begin{cases}
-(r\varphi_{\alpha}'(r,a))'=r(1+r^2)^{\alpha}e^{v_{\alpha}(r,a)}\varphi_{\alpha}(r,a),\quad r>0,\\
\varphi_{\alpha}(0,a)=1,\quad \varphi_{\alpha}'(0,a)=0,
\end{cases}
\end{equation}
and
\begin{equation}
\label{n8.9}
\beta_{\alpha}'(a)=\int_0^{+\infty}r(1+r^2)^{\alpha}e^{v_{\alpha}(r,a)}\varphi_{\alpha}(r,a)\mathrm{d}r.
\end{equation}
Similarly,
\begin{equation}
\label{n8.10}
\hat\varphi_{\alpha}(r,a)=\varphi_{\alpha}(\frac1r,a)
\end{equation}
satisfies
\begin{equation}
\label{n8.11}
-(r\hat\varphi_{\alpha}'(r,a))'=r^{\beta_{\alpha}(a)-2(\alpha+2)+1}(1+r^2)^{\alpha}e^{\hat v_{\alpha}(r,a)}\hat\varphi_{\alpha}(r,a),~ r>0.
\end{equation}
From \cite{l0} we know that,
\begin{equation}
\label{n8.12}
\begin{aligned}
&\varphi_{\alpha}(r,a)~\mbox{vanishes at least twice and it is bounded}\\
&\mbox{(with a finite limit at $+\infty$) if and only if}~ \beta_{\alpha}'(a)=0.
\end{aligned}
\end{equation}
Clearly, $\varphi_{\alpha}(r,a)$ is analytic with respect to $r$, and from (\ref{n8.8}) we can check easily that:
\begin{itemize}
  \item if $0<r_{a,\alpha}<R_{a,\alpha}$ are the first and last zero of $\varphi_{\alpha}(r,a)$ and $0<r^*_{a,\alpha}\leq R^*_{a,\alpha}$ are the first and last zero of $\varphi'_{\alpha}(r,a)$, then
\begin{equation}
\label{n8.13}
r_{a,\alpha}<r^*_{a,\alpha}\leq R_{a,\alpha}^*<R_{a,\alpha}.
\end{equation}
\end{itemize}
Furthermore, by means of Alexandroff-Bol isoperimetric inequality, it is proved in \cite[Lemma 3.3]{l0} that,
\begin{equation}
\label{n8.14}
\int_0^{r_{a,\alpha}^*}r(1+r^2)^{\alpha}e^{v_{\alpha}(r,a)}dr>4,\quad \int_{R_{a,\alpha}^*}^{+\infty}r(1+r^2)^{\alpha}e^{v_{\alpha}(r,a)}dr>2\beta_{\alpha}(a)-4(\alpha+1).
\end{equation}

Next, we use the additional information (specific to problem \eqref{new13}) that the value: $\beta=2(\alpha+2)$ is attained by the radial function: $v(r)=\log(\frac{4(\alpha+2)}{(1+r^2)^{\alpha+2}})$, which satisfies (\ref{n8.2}) with
\begin{equation}
\label{n8.16}
a=\log(4(\alpha+2))=:a_{\alpha}.
\end{equation}
In other words, we have:
\begin{equation}
\label{n8.15}
v_{\alpha}(r,a_{\alpha})=\log\left(\frac{4(\alpha+2)}{(1+r^2)^{\alpha+2}}\right) ~\mathrm{and}~\beta_{\alpha}(a_{\alpha})=2(\alpha+2).
\end{equation}

Since, $2(\alpha+2)<4\alpha$ if and only if $\alpha>2,$ by virtue of (\ref{n8.5}) and the continuity of $\beta_{\alpha}(a),$ we find that for $\alpha>2$ there holds:
\begin{equation}
\label{n8.17}
\exists~
a_{\alpha}^*\in\mathbb{R}:~\beta_{\alpha}(a_{\alpha}^*)=\min_{a\in\mathbb{R}}\beta_{\alpha}(a)\in(2(\alpha+1),4\alpha).
\end{equation}
Furthermore, we recall from \cite{l0} that we have:
\begin{equation}
\label{n8.18}
\beta_{\alpha}'(a)<0,~\forall a\in\beta_{\alpha}^{-1}([4\alpha,4(\alpha+1)))~\mathrm{and}~\alpha>1.
\end{equation}
By (\ref{n8.18}), in particular we see that (\ref{n8.17}) remains valid also for $\alpha=2$, where $2(\alpha+2)=4\alpha$, see also \cite{det} for a discussion of this situation. Therefore, it remains to show that (\ref{n8.17}) actually holds also when $\alpha\in(1,2).$ To this purpose, we let
\begin{equation}
\label{n8.19}
\Lambda:=\left\{\alpha\in(1,2):~(\ref{n8.17})~\mbox{holds}\right\}.
\end{equation}
By means of a continuity argument (with respect to $\alpha$) we can check that $\Lambda$ is not empty (by using continuity at $\alpha=2$) and is open. So, in order to obtain that $\Lambda=(1,2)$, it suffices to show that $\Lambda$ is closed relatively to the set $(1,2).$ Namely, we need to prove that
\begin{equation}
\label{n8.20}
\mbox{if}~\alpha_n\in\Lambda~\mathrm{and}~\alpha_n\to\alpha\in(1,2)~\mbox{then}~\alpha\in\Lambda.
\end{equation}
To establish (\ref{n8.20}) we let
\begin{equation}
\label{n8.21}
a_n=a_{\alpha_n}^*,\quad\beta_n:=\beta_{\alpha_n}(a_{\alpha_n}^*)=\min_{a\in\mathbb{R}}\beta_{\alpha_n}(a)\in(2(\alpha_n+1),4\alpha_n).
\end{equation}
Since
\begin{equation}
\label{n8.22}
\beta_{\alpha_n}'(a_n)=0,
\end{equation}
we know that
\begin{equation}
\label{n8.23}
\varphi_n(r)=\varphi_{\alpha_n}(r,a_n)
\end{equation}
defines a bounded solution of the linearized problem (\ref{n8.8}) around the solution $v_n(r):=v_{\alpha_n}(r,a_n)$. In particular we know that $\varphi_n(r)$ admits a finite limit as $r\to+\infty$, and by setting
\begin{equation}
\label{n8.24}
\varphi_n(\infty)=\lim_{r\to+\infty}\varphi_n(r)\quad\mathrm{and}\quad\hat\varphi_n(r)=\varphi_n(1/r),
\end{equation}
we find that
\begin{equation}
\label{n8.25}
\begin{cases}
-(r\hat\varphi_n'(r))=r^{\beta_n-2(\alpha_n+2)+1}(1+r^2)^{\alpha}e^{\hat v_n(r)}\hat\varphi_n(r),~\forall r>0,\\
\hat\varphi_n(0)=\varphi_n(\infty),~\quad\hat\varphi_n'(0)=0.
\end{cases}
\end{equation}
with $\hat v_n(r):=\hat v_{\alpha_n}(r,a_n).$ Furthermore, by virtue of (\ref{n8.13}) and (\ref{n8.14}) we know that, for the first zero $r_n$ and last zero $R_n$ of $\varphi_n$ and the first zero $r_n^*$ and last zero $R_n^*$ of $\varphi_n'$ the following holds:
\begin{equation}
\label{n8.26}
0<r_n<r_n^*\leq R_n^*<R_n,
\end{equation}
and
\begin{equation}
\label{n8.27}
\int_0^{r_n^*}r(1+r^2)^{\alpha}e^{v_n(r)}dr>4.
\end{equation}
To establish (\ref{n8.20}) it suffices to show that,
\begin{equation}
\label{n8.28}
a_n~\mbox{is uniformly bounded.}
\end{equation}
Indeed, if (\ref{n8.28}) holds, then along a subsequence, we have
\begin{equation*}
a_n\to a_0,~\mathrm{as}~n\to+\infty,
\end{equation*}
and by the uniform continuity of $\beta_{\alpha}(a)$ and $\beta_{\alpha}'(a)$ on compact sets of the parameters $(\alpha,a)$, we obtain that
$$\beta_{\alpha_n}(a_n)\to\beta_{\alpha}(a_0)=\min_{a\in\mathbb{R}}\beta_{\alpha}(a)~\mathrm{as}~n\to+\infty,$$
and
$$\beta_{\alpha}'(a_0)=0.$$
So, we can use (\ref{n8.4}) and (\ref{n8.18}) to deduce that $\beta_{\alpha}(a_0)\in(2(\alpha+1),4\alpha)$ and thus $\alpha\in\Lambda,$ as claimed.

Now it remains to prove (\ref{n8.28}). To this purpose, we start by observing that, for $1<\alpha<2$, the value $a_{\alpha}:=\log(4(1+\alpha))\in\beta_{\alpha}^{-1}((4\alpha,4(\alpha+1)))$, (see (\ref{n8.15}) and (\ref{n8.16}))
and so we can use (\ref{n8.18}) to derive that $\beta_{\alpha}(a)$ is strictly decreasing in $(-\infty,a_{\alpha}]$. As a consequence, for $\alpha\in(1,2),$ if $a\in\mathbb{R}$ and $\beta_{\alpha}(a)<4\alpha<2(\alpha+2)$ then necessarily $a>a_{\alpha}=\log(4(1+\alpha))$. Therefore, $a_n=a_{\alpha_n}^*>\log4(1+\alpha_n),~\forall n\in\mathbb{N}$, and so $a_n$ is uniformly bounded from below. To check that $a_n$ is also bounded from above, we argue by contradiction and assume that (along a subsequence):
\begin{equation}
\label{n8.29}
a_n\to+\infty.
\end{equation}
Thus, by Remark 8.1, we know that both $v_n$ and $\hat v_n$ admit a blow-up point at the origin (and only at the origin), (in the sense of \cite{bt1,bm,ls}) and in particular
\begin{equation}
\label{n8.30}
\hat v_n(0)\to+\infty,~\mathrm{and}~\beta_n\to4\alpha~\mathrm{as}~n\to+\infty.
\end{equation}
Therefore, by letting
\begin{equation}
\label{n8.31}
\varepsilon_n=e^{-\frac{v_n(0)}{2}}=e^{-\frac{a_n}{2}}\to 0~\mathrm{and}~\hat\varepsilon_n=e^{-\frac{\hat v_n(0)}{\beta_n-2(\alpha_n+1)}}\to0,~\mathrm{as}~n\to+\infty;
\end{equation}
for the scaled functions:
\begin{equation}
\label{n8.32}
V_n(r)=v_n(\varepsilon_nr)+2\log\varepsilon_n~\mathrm{and}~\hat V_n(r)=\hat v_n(\hat\varepsilon_nr)+(\beta_n-2(\alpha_n+1))\log\hat\varepsilon_n,
\end{equation}
we find that
\begin{equation}
\label{8.33}
\begin{cases}
-(rV_n'(r))'=r(1+\varepsilon^2_nr^2)^{\alpha}e^{V_n(r)},~\quad r>0,\\
0=V_n(0)=\max_{r\geq0}V_n(r),\\
\beta_n=\int_0^{\infty}(1+\varepsilon^2_nr^2)e^{V_n(r)}r\mathrm{d}r,
\end{cases}
\end{equation}
and
\begin{equation}
\label{n8.34}
\begin{cases}
-(r\hat V_n'(r))'=r^{\beta_n-2(\alpha_n+2)+1}(1+\hat\varepsilon^2_nr^2)^{\alpha}e^{\hat V_n(r)},~\quad r>0,\\
0=\hat V_n(0)=\max_{r\geq0}\hat V_n(r),\\
\beta_n=\int_0^{\infty}r^{\beta_n-2(\alpha_n+2)+1}(1+\hat\varepsilon^2_nr^2)e^{\hat V_n(r)}r\mathrm{d}r.
\end{cases}
\end{equation}
Thus, we can use well known Harnack type inequality, (as in \cite{bt1,bm,ls}) to conclude that (along a subsequence)
$$V_n\to V~\mbox{uniformly in}~C_{\mathrm{loc}}^2,~\mathrm{as}~n\to+\infty$$
with $V(r)=\log\frac{1}{(1+\frac18r^2)^2}$, the unique radial solution of the Liouville equation:
\begin{equation}
\label{n8.35}
\begin{cases}
-(rV'(r))'=re^{V(r)},~r>0,\\
0=V(0)=\max_{r\geq0} V(r),~\int_0^{+\infty}e^{V(r)}r\mathrm{d}r<+\infty;
\end{cases}
\end{equation}
and
$$\hat V_n(r)\to\hat V~\mbox{uniformly in}~C_{\mathrm{loc}}^{1,\gamma},~\mathrm{as}~n\to+\infty$$
with $\hat V(r)=\log\frac{1}{(1+\frac{1}{8(\alpha-1)^2}r^{2(\alpha-1)})^2}$, the unique radial solution of the (singular) Liouville equation:
\begin{equation}
\label{n8.36}
\begin{cases}
-(r\hat V'(r))'=r^{1-2(2-\alpha)}e^{\hat V(r)},~r>0,\\
0=\hat V(0)=\max_{r\geq0} \hat V(r),~\int_0^{+\infty}r^{1-2(2-\alpha)}e^{\hat V(r)}\mathrm{d}r<+\infty.
\end{cases}
\end{equation}
Since $\int_0^{+\infty}re^{V(r)}\mathrm{d}r=4$ (see \cite{ls}), from (\ref{n8.27}) we derive in particular that,
\begin{equation}
\label{n8.37}
\frac{r_n^*}{\varepsilon_n}\to+\infty,~\mathrm{as}~n\to+\infty,
\end{equation}
where $r_n^*$ is the first zero of $\varphi_n'.$

Next to compare the blow up rates $\varepsilon_n$ and $\hat\varepsilon_n$, we recall the following profile estimates, valid for $v_n$ and $\hat v_n$ respectively:
\begin{equation}
\label{n8.38}
\left|v_n(r)-\log\frac{e^{v_n(0)}}{(1+\frac{e^{v_n(0)}}{8}r^2)^2}\right|\leq C, ~\forall r: 0\leq r\leq 1;
\end{equation}
and
\begin{equation}
\label{n8.39}
\left|\hat v_n(r)-\log\frac{e^{\hat v_n(0)}}{(1+\frac{e^{\hat v_n(0)}}{8(\beta_n-2(\alpha_n+1))}r^{\beta_n-2(\alpha_n+1)})^2}\right|\leq C, ~\forall r: 0\leq r\leq 1;
\end{equation}
with a suitable constant $C>0$. We notice that, for negative powers in (\ref{n8.7}), as it occurs in our situation, a detailed proof of the pointwise estimate (\ref{n8.30}) can be found in \cite{bt2}.

Using (\ref{n8.38}) and (\ref{n8.39}) with $r=1$, we deduce the following:
\begin{equation*}
a_n=v_n(0)=\hat v_n(0)+O(1),~\forall n\in\mathbb{N}.
\end{equation*}
As a consequence, we can estimate
\begin{equation}
\label{n8.40}
\frac{\hat\varepsilon_n}{\varepsilon_n}=\frac{e^{\frac{v_n(0)}{2}}}{e^{\frac{\hat v_n(0)}{\beta_n-2(\alpha_n+1)}}}\leq Ce^{\frac{\beta_n-2(\alpha_n+2)}{2(\beta_n-2(\alpha_n+1))}v_n(0)}
\end{equation}
with $C>0$ a suitable constant. Thus, by recalling that: $\alpha_n\to\alpha\in(1,2)$ and $\beta_n\to4\alpha$, from (\ref{n8.40}) we derive:
\begin{equation}
\label{n8.41}
\frac{\hat\varepsilon_n}{\varepsilon_n}\to 0,~\mathrm{as}~n\to+\infty.
\end{equation}
At this point we can use the information about the linearized problem around $v_n$ and $\hat v_n$. To this purpose, by recalling that: $\varphi_n(\infty)=\lim_{r\to+\infty}\varphi_n(r)\in\mathbb{R},$ we can define $t_n\in[0,+\infty]$ such that:
\begin{equation}
\label{n8.42}
\max_{r\geq0}|\varphi_n(r)|=|\varphi_n(t_n)|.
\end{equation}
We set
\begin{equation}
\label{n8.43}
\psi_n(r)=\frac{\varphi_n(\varepsilon_nr)}{|\varphi_n(t_n)|}~\mathrm{and}~\hat\psi_n(r)=\frac{\varphi_n(\frac{\hat\varepsilon_n}{r})}{|\varphi_n(t_n)|}
\end{equation}
and notice that $\hat\psi_n$ extends smoothly at $r=0$. Furthermore, the following holds:
\begin{equation*}
\begin{cases}
-(r\psi_n'(r))'=r(1+\varepsilon_n^2r^2)^{\alpha}e^{V_n(r)}\psi_n(r),~r>0,\\
|\psi_n(r)|\leq 1,~\forall r\geq0,
\end{cases}
\end{equation*}
and
\begin{equation*}
\begin{cases}
-(r\hat\psi_n'(r))'=r^{\beta_n-2(\alpha_n+2)+1}(1+\hat\varepsilon_n^2r^2)^{\alpha}e^{\hat V_n(r)}\hat\psi_n(r),~r>0,\\
|\hat\psi_n(r)|\leq 1,~\forall r\geq0.
\end{cases}
\end{equation*}
As a consequence, along a subsequence, we find that
\begin{equation}
\label{n8.44}
\psi_n\to\psi~\mathrm{and}~\hat\psi_n\to\hat\psi,~\mathrm{as}~n\to+\infty
\end{equation}
uniformly in $C_{\mathrm{loc}}^2$, with $\psi$ and $\hat\psi$ respectively bounded radial solution of the linearized problem around the solution: $V(r)=\log\frac{1}{(1+\frac{r^2}{8})^2}$ of problem (\ref{n8.35}) and of the linearized problem around the solution: $\hat V(r)=\log\frac{1}{(1+\frac{1}{8(\alpha-1)^2}r^{2(\alpha-1)})^2}$ of problem (\ref{n8.36}). Therefore, we have:
\begin{equation}
\label{n8.45}
\begin{aligned}
\psi(r)=\mu\left(\frac{1-\frac18r^2}{1+\frac18r^2}\right),~\mu\in\mathbb{R},\quad\mathrm{and}\quad
\hat\psi(r)=\hat\mu\left(\frac{1-\frac{r^{2(\alpha-1)}}{8(\alpha-1)^2}}{1+\frac{r^{2(\alpha-1)}}{8(\alpha-1)^2}}\right),~\hat\mu\in\mathbb{R}.
\end{aligned}
\end{equation}
In the following, we divide our discussion into two cases.
\medskip

\noindent Case (1). $t_n\neq0,$ then necessarily either $t_n=+\infty$ or $t_n\in(0,+\infty)$ and $\varphi_n'(t_n)=0.$ In particular, in the later case we have (see (\ref{n8.26})):
\begin{equation}
\label{n8.46}
0<r_n\leq r_n^*\leq t_n\leq R_n^*<R_n,
\end{equation}
where $r_n$ and $R_n$ are the first and last zero of $\varphi_n.$ Therefore, if we set
\begin{equation*}
\hat r_n=
\begin{cases}
0,~&\mathrm{if}~t_n=+\infty,\\
\frac{\hat\varepsilon_n}{t_n},~&\mathrm{if}~t_n\in(0,+\infty),
\end{cases}
\end{equation*}
then, by virtue of (\ref{n8.37}), (\ref{n8.41}) and (\ref{n8.46}) we find that $\hat r_n\to0$ as $n\to+\infty$, and moreover by (\ref{n8.43}) we have: $|\hat\psi_n(\hat r_n)|=1.$
Consequently,
\begin{equation}
\label{n8.47}
|\hat\psi(0)|=1.
\end{equation}
On the other hand, by (\ref{n8.37}), (\ref{n8.41}) and (\ref{n8.46}) we see also that for the last zero $R_n$ of $\varphi_n$ there holds:
\begin{equation*}
\hat\psi_n(\frac{\hat\varepsilon_n}{R_n})=0~\mathrm{and}~\frac{\hat\varepsilon_n}{R_n}\to0~\mathrm{as}~n\to+\infty.
\end{equation*}
This implies that $\hat\psi(0)=0$, in contradiction to (\ref{n8.47}).
\medskip

\noindent Case (2). $t_n=0$, then we have
$$1=\varphi_n(0)=\max_{r\geq0}|\varphi_n(r)|;$$
and so
\begin{equation}
\label{n8.48}
\psi_n(r)=\varphi(\varepsilon_nr)\to\psi(r)=\frac{1-\frac18r^2}{1+\frac18r^2},~\mathrm{as}~n\to+\infty
\end{equation}
uniformly in $C_{\mathrm{loc}}^2.$ In fact, if we let $r_0>0:~r_0^{2(\alpha-1)}=8(\alpha-1)$, then in view of (\ref{n8.44}) and (\ref{n8.45}) we find that
$$\hat\psi_n(r_0)\to0~\mathrm{as}~n\to+\infty.$$
While, by (\ref{n8.41}), (\ref{n8.43}) and (\ref{n8.48}) we have:
\begin{equation*}
\hat\psi_n(r_0)=\varphi_n(\frac{\hat\varepsilon_n}{r_0})=\psi_n(\frac{\hat\varepsilon_n}{\varepsilon_n}\frac{1}{r_0})\to1,~\mathrm{as}~n\to+\infty,
\end{equation*}
and we reach again a contradiction. Therefore (\ref{n8.29}) is not true and (\ref{n8.28}) is established. As a consequence, for all $\alpha>1$, (\ref{n8.17}) holds and problem (\ref{new13}) admits a radial solution if and only if
\begin{equation*}
\rho\in[\bar\rho_{\alpha},8\pi(\alpha+1)),~\mathrm{with}~\bar\rho_{\alpha}=2\pi\beta_{\alpha}(a_{\alpha}^*).
\end{equation*}
In addition by virtue of (\ref{n8.18}) (see \cite{l1}), we know that problem (\ref{new13}) admits a unique radial solution for $\rho\in[8\pi\alpha,8\pi(\alpha+1))$. While for $\rho\in(\bar\rho_{\alpha},8\pi\alpha)$ there exists at least two values:
\begin{equation*}
a_1<a_{\alpha}^*<a_2~\mathrm{such~ that}~2\pi\beta_{\alpha}(a_1)=\rho=2\pi\beta_{\alpha}(a_2);
\end{equation*}
So we obtain at least two radial solutions for (\ref{new13}) in this case, and the proof is completed. \hfill $\square$
\medskip

It is an interesting open question to see whether problem (\ref{new13}) admits no solution (i.e. not necessarily radial) for $\rho\in(4\pi(\alpha+1),\bar\rho_{\alpha}).$
\medskip

We point out that the results we have used from \cite{gm} and \cite{l0} apply to more general Liouville-type problems of the type:
\begin{equation}
\label{n8.49}
-\Delta v=K(|x|)e^{v}~\mathrm{in}~\mathbb{R}^2,\quad \rho=\int_{\mathbb{R}^2}K(|x|)e^v
\end{equation}
with a weight function $K>0$ satisfying:
\begin{equation}
\label{n8.50}
\Delta\log K\geq 0~\mathrm{and}~\lim_{r\to+\infty}\frac{rK'(r)}{K(r)}=2\alpha.
\end{equation}
In other words, properties (i) and (ii) of Proposition \ref{pr2.new} continue to hold for problem (\ref{n8.49})-( \ref{n8.50}). On the contrary, part (iii) is specific of problem (\ref{new13}), namely when we choose $K(r)=(1+r^2)^\alpha$. For example, we mention that if we consider the (apparently) similar weight function: $K(r)=(1+r^{2\alpha})$, then for $\alpha>1$, problem (\ref{n8.49}) admits a radial solution if and only if $\rho\in(8\pi\alpha,8\pi(\alpha+1))$, in striking contrast to (iii) of Proposition \ref{pr2.new}. We refer the readers to \cite{pt} for details.

\end{document}